\documentclass[aip,jmp,author-numerical,reprint]{revtex4-1}
\usepackage[T1]{fontenc}
\usepackage{graphicx}
\usepackage{amssymb,amsthm,amsmath,mathrsfs,bm,braket,marginnote}
\usepackage{enumerate}
\usepackage{appendix}
\usepackage[colorlinks=true, pdfstartview=FitV, linkcolor=blue, citecolor=blue, urlcolor=blue]{hyperref}
\usepackage{multirow}
\usepackage[utf8]{inputenc}
\usepackage[normalem]{ulem}

\usepackage{xcolor}
\usepackage{pgf}
\usepackage{pgfplots}
\usepackage{tikz}

\numberwithin{equation}{section}
\newtheorem{theorem}{Theorem}[section]
\newtheorem{lemma}[theorem]{Lemma}

\newtheorem{corollary}[theorem]{Corollary}

\newtheorem{proposition}[theorem]{Proposition}

\theoremstyle{remark}
\newtheorem{remark}[theorem]{Remark}

 \reversemarginpar

\newcommand{\Z}{\mathbb{Z}}

\newcommand{\ex}{\epsilon_x}

\newcommand{\derivative}[1]{{d \over d{#1} }}
\newcommand{\N}{\mathbb{N}}

\renewcommand{\theequation}{\arabic{section}.\arabic{equation}}
 
\begin{document}

\title[Classical $\beta$-ensembles with $\beta$ proportional to $1/N$]{The classical $\beta$-ensembles with $\beta$ proportional to $1/N$: from loop equations to Dyson's disordered chain}

\author{P.J. Forrester}
\affiliation{School of Mathematical and Statistics, ARC Centre of Excellence for Mathematical and Statistical Frontiers, The University of Melbourne, Victoria 3010, Australia}
\email{pjforr@unimelb.edu.au}

\author{G. Mazzuca}
\affiliation{International School for Advanced Studies (SISSA), Via Bonomea 265, 34136 Trieste, Italy}
\email{guido.mazzuca@sissa.it}

\date{\today}

\begin{abstract}
In the classical $\beta$-ensembles of random matrix theory, setting $\beta = 2 \alpha/N$ and taking the $N \to \infty$ limit gives a statistical state depending on $\alpha$. Using the loop equations for the
classical $\beta$-ensembles, we study the corresponding eigenvalue density, its moments, covariances of monomial linear statistics, and the moments of the leading $1/N$ correction to the density. From earlier literature, the limiting eigenvalue density is known to be related to classical functions. Our study gives a unifying mechanism underlying this fact, identifying, in particular, the Gauss hypergeometric differential equation determining the Stieltjes transform of the limiting density in the Jacobi case. Our characterisation of the moments and covariances of monomial linear statistics is through recurrence relations. Also, we extend recent work which begins with the $\beta$-ensembles in the high-temperature limit  and constructs a family of tridiagonal matrices referred to as $\alpha$-ensembles, obtaining a random anti-symmetric tridiagonal matrix with  i.i.d.~gamma distributed random variables. From this we can supplement analytic results obtained by Dyson in the study of the so-called type I disordered chain.
\end{abstract}

\maketitle

\section{Introduction}
The Gaussian $\beta$-ensemble refers to the eigenvalue probability density function (PDF)
proportional to 
\begin{equation}\label{0.1}
\prod_{l=1}^N e^{-\lambda_l^2} \,\prod_{1\leq j<k\leq N}|\lambda_k-\lambda_j|^{\beta}.
\end{equation}
Upon the scaling of the eigenvalues by setting
\begin{equation}\label{0.2}
\lambda_l =   \sqrt{\beta N} x_l,
\end{equation}
it is a well known fact that the eigenvalue density $\bar{\rho}_{(1)}(x)$, normalised to
integrate to unity, has the limiting form of the Wigner semi-circle law (see e.g.\cite[\S 1.4.2]{Fo10})
\begin{equation}\label{0.3}
\lim_{N \to \infty} \bar{\rho}_{(1)}(x) = {2 \over \pi} (1 - x^2)^{1/2} \chi_{|x| < 1},red{,}
\end{equation}
where $\chi_A = 1$ for $A$ true and
$\chi_A = 0$ otherwise. 
 \color{black}
 The use of the scaling variables (\ref{0.2}) --- often referred to as corresponding
 to the global regime; see e.g.~\cite{Fo20} --- also leads to many other consequences.
 For example, introduce the linear statistic $A = \sum_{j=1}^N a(x_j)$ for $a(x)$ smooth and
 bounded. The average with respect to (\ref{0.1}) then permits
 the $1/N$ expansion \cite{Jo98}
\begin{equation}\label{0.4} 
\Big \langle  \sum_{j=1}^N a(x_j) \Big \rangle  = N \int_{-\infty}^\infty  a(x) \bar{\rho}_{(1),0}(x) \, {\rm d}x +
  \int_{-\infty}^\infty  a(x) \bar{\rho}_{(1),1}(x) \, {\rm d}x + {\rm O} \Big ( {1 \over N } \Big ),
\end{equation}
 where
\begin{equation}\label{0.5} 
\bar{\rho}_{(1),1}(x) = \Big ( {1 \over \beta} - {1 \over 2} \Big )
\Big ( {1 \over 2} ( \delta ( x - 1) + \delta (x + 1) ) - {1 \over \pi \sqrt{1 - x^2}} \Big ).
\end{equation}

Equivalently, the smoothed eigenvalue density (i.e.~effective eigenvalue density upon integrating over a smooth test function), $\bar{\rho}_{(1)}^{\rm s}(x)$ say,
admits an expansion in $1/N$ powers, 
\begin{equation}
    \bar{\rho}_{(1)}^{\rm s}(x) = \sum_{j=0}^\infty \bar{\rho}_{(1),j}(x) N^{-j}, 
\end{equation}
where the first two terms are given by (\ref{0.3}) and
(\ref{0.5}) respectively.

Furthermore (see e.g.~\cite{PS11})
\begin{equation}\label{0.6} 
\lim_{N \to \infty} {\rm Var} \, A = \int_{-1}^1 {\rm d}x   \int_{-1}^1 {\rm d}y \,
\Big ( {a(x) - a(y) \over    x - y} \Big )^2   r_{(2),0}(x,y),
\end{equation}
where
\begin{equation}\label{0.7} 
  r_{(2),0}(x,y) = {1 \over 4 \pi^2} 
 {1 - xy \over \sqrt{1 - x^2} \sqrt{1 - y^2}}.
 \end{equation}
 
 The global regime is characterised by the spacing between the eigenvalues tending to zero,
 at such a rate that the statistical properties like those reviewed in the above paragraph have a
 well defined limit. This latter property is also shared by another choice of limit, corresponding to
 a scaled high temperature regime as specified by setting
 \begin{equation}\label{0.8}  
\beta = 2 \alpha /N, \qquad \alpha  > -1 \: \: {\rm fixed},
\end{equation}
before taking the large $N$ limit. The study of this limit was
introduced in the context of the Gaussian $\beta$-ensemble in \cite{ABG12}. 
Later is was considered for the Laguerre and Jacobi variants of (\ref{0.1})
\cite{ABMV13,TT19,TT20} , i.e.~the
primary examples of the classical ensembles in random matrix theory. Related to the $\beta$-ensembles with the scaling (\ref{0.8}) are certain classes of random tridiagonal matrices, with i.i.d.~entries along the diagonal, and (separately) along the leading diagonal,
now referred to as specifying $\alpha$-ensembles \cite{Ma20}.

After making
the $N$, $\beta$-independent change of scale $\lambda_j = x_j/\sqrt{2}$ in (\ref{0.1}), 
upon the limit (\ref{0.8}) 
 the density $\rho_{(1),0}(x) = \rho_{(1),0}(x;\alpha)$
is specified by the functional form
\cite{ABG12, DS15, Ma20}
\begin{equation}\label{0.8+}  
 \rho_{(1),0}(x;\alpha) = {e^{- x^2/2} \over \sqrt{2 \pi}} | \hat{f}_\alpha(x) |^{-2}, \qquad   \hat{f}_\alpha(x) = \sqrt{c \over \Gamma(\alpha)}
 \int_0^\infty t^{\alpha-1} e^{-t^2/2} e^{i x t} \, dt.
\end{equation} 
While it is to be anticipated that a $1/N$ expansion of the form (\ref{0.4}) will again hold --- and thus with the
first term known by way of (\ref{0.8+}) the task remaining is to characterise the analogue for
$ \bar{\rho}_{(1),1}(x)$ --- results from \cite{Tr17,NT18} (see also \cite{HL20}) tell us that in relation to the
variance
\begin{equation}\label{0.9+}  
{1 \over N} {\rm Var} \, A
\end{equation} 
has a well defined limit. Note that the factor of $1/N$ is absent on the LHS of (\ref{0.6}). However it remains to obtain
explicit formulas in relation to $A$.

In the present work we introduce a new approach --- making use of knowledge of the loop equations for the classical
$\beta$-ensembles \cite{BEMN10, BMS11, MMPS12, WF14, FRW17} --- to systematically study the high temperature scaling (\ref{0.8}). Choosing  the Gaussian $\beta$-ensemble for definiteness in this Introduction, the loop equation formalism allows for a systematic quantification of the quantities in the large $N$-expansion
\begin{equation}\label{WG1}
    {1 \over N} \Big \langle \sum_{j=1}^N {1 \over x - \lambda_i} \Big \rangle^{\rm G} \Big |_{\beta = 2 \alpha/N} = W_1^{0, \rm G}(x) + {1 \over N}W_1^{1, \rm G}(x) + \cdots,
\end{equation}
 where
\begin{equation}\label{WG2} 
W_1^{0, \rm G}(x) = \int_{-\infty}^\infty {\rho_{(1),0}^{\rm G} (\lambda; \alpha) \over x - \lambda} \, d \lambda, \qquad
W_1^{1, \rm G}(x) = \int_{-\infty}^\infty {\rho_{(1),1}^{\rm G} (\lambda; \alpha) \over x - \lambda} \, d \lambda,
\end{equation}
as well as in the large $N$-expansion
\begin{equation}\label{WG3} 
{1 \over N} {\rm Cov} \, \Big ( \sum_{i=1}^N {1 \over x - \lambda_i},
\sum_{i=1}^N {1 \over y - \lambda_i} \Big )^{\rm G}
\Big |_{\beta = 2 \alpha/N} :=
W_2^{0, \rm G}(x,y) + {1 \over N}W_2^{1, \rm G}(x,y)
+ \cdots,
\end{equation}
where Cov$\,( \cdot, \cdot)$ denotes the covariance of the respective linear statistics, and we use the superscript ``G'' to indicate the Gaussian ensemble in the scaling limit with $\beta =2 \alpha/N$.
Note that from knowledge of $W_1^{0, \rm G}(x)$ as specified in (\ref{WG2}), which is the Stieltjes transform of $\rho_{(1),0}^{\rm G} (\lambda; \alpha)$,
the corresponding inversion formula gives
\begin{equation}\label{u4}
\rho_{(1),0}^{\rm G}(x;\alpha) =  \lim_{\epsilon \to 0^+} {1 \over \pi}  {\rm Im} \,  W_1^{0,\rm G}(x - i \epsilon).
\end{equation}

For the classical ensembles generally, we will show that the loop equation formalism implies $W_1^0$ can be computed as the solution of a differential equation. This fact is already known in the Gaussian and Laguerre cases, but not for the Jacobi ensemble. This then allows for the computation of the leading order scaled density via (\ref{u4}). The differential equation characterisation also allows for the corresponding moments of the spectral density to  be determined via a recurrence. Again specialising to the Gaussian ensemble for definiteness, we see from performing an appropriate geometric series expansion in the first expression of (\ref{WG2}) that
\begin{equation}\label{v1}
 W_1^{0, \rm G}(x) = {1 \over x} \sum_{p=0}^\infty {m_{p,0}^{\rm G} \over x^p }, \qquad m_{p,0}^{\rm G} = \int_{-\infty}^\infty x^p  \rho_{(1),0}^{\rm G}(x;\alpha)  \, dx,
\end{equation}
where the formula for $ m_{p,0}^{\rm G}$ in terms of $\rho_{(1),0}^{\rm G}(x;\alpha) $ tells us that $\{ m_{p,0}^{\rm G} \}$ are the moments of the limiting 
eigenvalue density. 

\begin{proposition}\label{P01}
(Duy and Shirai \cite[Prop.~3.1]{DS15}.) The moments
$\{ m_{p,0}^{\rm G} \}_{p \: {\rm even}}$ satisfy the recurrence
\begin{equation}\label{v2}
m_{p+2,0}^{\rm G} = (p+1) m_{p,0}^{\rm G} + \alpha \sum_{s=0}^{p/2} m_{p-2s,0}^{\rm G} m_{2s,0}^{\rm G}, \qquad m_{0,0}^{\rm G} = 1, \: \: (p=0,2,4,\dots {\rm even}),
\end{equation}
while the odd moments all vanish by symmetry.
\end{proposition}

The loop equation formalism shows that $W_{2,0}^{\rm G}(x_1,x_2)$ satisfies a partial differential equation involving $W_{1,0}^{\rm G}(x_i)$ ($i=1,2$). No closed form solution is to be expected, but analogous to (\ref{v1}) if we expand about infinity by noting
\begin{equation}\label{v3}
W_2^{0, \rm G}(x_1,x_2) = {1 \over x_1 x_2}  \sum_{p,q=0}^\infty   {   \mu_{(p,q),0}^{\rm G}  \over x_1^p x_2^q}, \qquad
 \mu_{(p,q),0}^{\rm G}  =    \lim_{N \to \infty \atop \beta = 2 \alpha/ N} {\rm Cov} \, \Big ( \sum_{i=1}^N x_i^p, \sum_{i=1}^N x_i^q \Big )^{\rm G},
 \end{equation}
 then the partial differential equation allows
 $\{ \mu_{(p,q),0}^{\rm G} \}$ to be determined by a coupled recurrence involving $\{ m_{p,0}^{\rm G} \}$, already determined by (\ref{v2}).

\begin{proposition}\label{P02}
(Equivalent to Spohn \cite[Eqns.~(5.14), (5.15)]{Sp20}.)
For $p,q \ge 1$ of the same parity, meaning that they are either both even or odd, we have
 \begin{equation}\label{vf}   
 \mu_{(p,q),0}^{\rm G} = (p-1) \mu_{(p-2,q),0}^{\rm G} + q m_{p+q-2,0}^{\rm G} + 2\alpha \sum_{s=0}^{\lfloor p/2 - 1 \rfloor}
 m_{2s,0}^{\rm G} \mu_{(p-2-2s,q),0}^{\rm G}.
  \end{equation}  
  If $p=0$, or $q=0$, or $p,q$ have the opposite parity,
  $\mu_{(p,q),0}^{\rm G} =0$.
\end{proposition}

Our study of $W_1^{1, \rm G}(x)$ proceeds analogously. Introducing the expansions
\begin{equation}\label{x3}   
  W_1^{1, \rm G}(x) = {1 \over x} \sum_{p=1}^\infty {m_{2p,1}^{\rm G} \over x^{2p}},
  \end{equation}
\begin{equation}\label{w1}
  W_2^{0,\rm G}(x,x) = {1 \over x^2} \sum_{p=1}^\infty {\tilde{\mu}_{2p,0}^{\rm G} \over x^{2p}}, \qquad
  \tilde{\mu}_{2p,0}^{\rm G}  = \sum_{p_1 + q_1 = 2p} \mu_{(p_1,q_1),0}^{\rm G},
 \end{equation}  
 from the loop equations we can determine that
 $\{m_{2p,1}^{\rm G}\}$ satisfies a coupled recurrence with $\{ \tilde{\mu}_{2p,0}^{\rm G} \}$ and
 $\{ m_{2p,0}^{\rm G} \}$.
 
 \begin{proposition}\label{P03}
 We have
  \begin{equation}\label{x4m}  
   m_{2(p+1),1}^{\rm G} = - \alpha (2p+1) m_{2p,0}^{\rm G} + (2p+1) m_{2p,1}^{\rm G} + \alpha \tilde{\mu}_{2p,0}^{\rm G}
   + 2 \alpha \sum_{s=0}^{p-1} m_{2s,0}^{\rm G} m_{2 (p-s), 1}^{\rm G}, \quad m_{0,1}^{\rm G} = 0,
  \end{equation} 
  where $\{ m_{2p,0}^{\rm G} \}$ are determined by the recurrence (\ref{v2}), and $\{ \tilde{\mu}_{2p,0}^{\rm G} \}$ by the recurrence (\ref{w1+}) below.
\end{proposition} 

After revising relevant results relating to the loop equation formalism for the classical ensembles in Section \ref{S2}, we proceed in Sections \ref{S3}, \ref{S4}, \ref{S5} respectively to derive Propositions \ref{P01}--\ref{P03} and their analogues for the Gaussian,
Laguerre and Jacobi $\beta$-ensembles with high temperature scaling (\ref{0.8}). Our strategy also gives a unifying method to derive the functional form of the limiting density, given by (\ref{0.8+}) in the Gaussian case; for the Jacobi ensemble this is new. Thus the loop equations give a particular Riccati equation for the Stieltjes transform of the limiting density, which implies a linear second order differential equation when the latter is written as a logarithmic derivative. In the Jacobi case, the linear second order differential equation is a hypergeometric differential equation, which leads to a functional form for the limiting density in terms of a linear combination of Gauss hypergeometric, in agreement with a recent result of Trinh and Trinh \cite{TT20}.

In the recent work, \cite{Ma20} the classical $\beta$-ensembles in the high temperature limit have been used to construct a family of tridiagonal matrices referred to as $\alpha$-ensembles. Moreover, an application was given to the study of generalised Gibbs ensembles associated with the classical Toda lattice \cite{Sp20}.
 In \S \ref{S6}, beginning with the anti-symmetric Gaussian $\beta$-ensemble we identify a further example of an $\alpha$-ensemble, specified as a random anti-symmetric tridiagonal matrix, with  i.i.d.~gamma distributed random variables. Knowledge of the limiting spectral density for the Laguerre $\beta$-ensemble in the scaled high temperature limit can be used to determine the limit spectral density of this particular $\alpha$-ensemble. It is pointed out that the same random matrix ensemble appears in Dyson's \cite{Dy53} study of a disordered chain of harmonic oscillators. Our analytic results supplement those already contained in Dyson's work.

\section{Preliminaries}\label{S2}
\subsection{Quantities of interest in the loop equation formalism}
Introduce the notation ME${}_{\beta,N}[w]$ to denote a matrix ensemble with eigenvalue PDF proportional to
\begin{equation}\label{1.2}
\prod_{l=1}^Nw(\lambda_l)\,\prod_{1\leq j<k\leq N}|\lambda_k-\lambda_j|^{\beta},
\end{equation}
where $w(\lambda)$ is referred to as the weight function. Collectively, the terminology $\beta$-ensemble
is used in relation to (\ref{1.2}), and the name associated with the weight is specified as an adjective.
Thus, for example, ME${}_{\beta,N}[e^{-\lambda^2}]$ is referred to as the Gaussian $\beta$-ensemble,
in agreement with the terminology used in relation to (\ref{0.1}).
Let $\rho_{(1)}(\lambda)$ denote the corresponding eigenvalue density, specified by the requirement that $\int_a^b\rho_{(1)}(\lambda) \, \mathrm{d}\lambda$ be equal to the expected number of eigenvalues in a general
interval $[a,b]$.  Its Stieltjes transform  is given by
\begin{equation}\label{2.1}
\overline{U}_1(x)=\int_{-\infty}^{\infty}\frac{\rho_{(1)}(\lambda)}{x-\lambda} \, \mathrm{d}\lambda.
\end{equation}
Note that
\begin{equation}\label{2.2}
\overline{U}_1(x)=\left\langle\sum_{j=1}^N\frac{1}{x-\lambda_j}\right\rangle_{{\rm ME}_{\beta,N}[w]} =
\left\langle {\rm Tr} \, ( x \mathbb I_N - H)^{-1} \right \rangle_{ {\rm ME}_{\beta,N}[w]},
\end{equation}
where in the second average $H = {\rm diag} \, (\lambda_1,\dots, \lambda_N)$.
In matrix theory $( x \mathbb I_N - H)^{-1} $ is referred to as the resolvent. It is thus by abuse of terminology that
$\overline{U}_1(x)$ itself is often referred to as the resolvent.
The average in the first equality in \eqref{2.2} is an example of a one-point correlator. 
Its generalisation to an $n$-point correlator  is
\begin{equation}\label{2.3}
\overline{W}_n(x_1,\ldots,x_n)=\left\langle\sum_{j_1,\ldots,j_n=1}^N\frac{1}{(x_1-\lambda_{j_1})\cdots(x_n-\lambda_{j_n})}\right\rangle_{{\rm ME}_{\beta,N}[w]}.
\end{equation}

A feature of (\ref{2.2}) is that for a large class of weights $w$, there is a scale $x = c_N s$ such that in the variable $s$ and as $N \to \infty$ the eigenvalue support is
a finite interval, and moreover $\overline{W}_1(c_N s)$ can be expanded as a series in $1/N$ \cite{BG12}
\begin{equation}\label{2.2a}
c_N \bar{W}_1(c_N s) = N \sum_{\ell}=0^\infty {\overline{W}_1^\ell(s) \over (N \sqrt{\kappa})^\ell}, \qquad \kappa = \beta/2,
\end{equation}
where $\{  \overline W_1^l(s) \}$ are independent of $N$. For example, from (\ref{0.2}), in the case of the Gaussian $\beta$-ensemble
$c_N = \sqrt{\beta N}$.
An analogous expansion holds true in relation to the $n$-point statistic (\ref{2.3}),
but only after forming appropriate linear combinations of
$\overline{W}_n$. These are  the connected components of $\overline{U}_n$, specified by 
\begin{align}
\overline{W}_1(x)&=\overline{U}_1(x)\nonumber
\\ \overline{W}_2(x_1,x_2)&=\overline{U}_2(x_1,x_2)-\overline{U}_1(x_1)\overline{U}_1(x_2)\nonumber
\\ \overline{W}_3(x_1,x_2,x_3)&=\overline{U}_3(x_1,x_2,x_3)-\overline{U}_2(x_1,x_2)\overline{U}_1(x_3)-\overline{U}_2(x_1,x_2)\overline{U}_1(x_2)\nonumber
\\&\quad-\overline{U}_2(x_2,x_3)\overline{U}_1(x_1)+2\overline{U}_1(x_1)\overline{U}_1(x_2)\overline{U}_1(x_3), \label{2.4}
\end{align}
with the general case $ \overline{W}_k$ being formed by an analogous inclusion/ exclusion construction.
Going in the reverse direction, and thus specifying $\{ \overline{U}_j \}$ in terms of $\{ \overline{W}_j \}$, the  inductive
relation
\begin{equation}
\overline{U}_n(x_1,J_n)=\overline{W}_n(x_1,J_n)+\sum_{\emptyset\neq J\subseteq J_n}\overline{W}_{n-|J|}(x_1,J_n\setminus J)\overline{U}_{|J|}(J),\label{2.6}
\end{equation}
where
\begin{equation}
J_n=(x_2,\ldots,x_n),\quad J_1=\emptyset, 
\end{equation}
 holds true (see e.g. \cite[pp. 8-9]{WF15}). The utility of the connected components $\overline{W}_n$ is that (\ref{2.2a})
 admits the generalisation \cite{BG12},
\begin{equation}\label{2.7a}
c_N^n\overline{W}_n(c_Ns_1,\ldots,c_Ns_n)=N^{2-n}\kappa^{1-n}\sum_{l=0}^{\infty}\frac{W_n^l(s_1,\ldots,s_n)}{(N\sqrt{\kappa})^l},
\end{equation}
where again $\kappa = \beta/2$. Thus as $n$ increases by one, the large $N$ form decreases by a factor of $1/N$, with all
lower order terms given by a series in $1/N$.

\subsection{Explicit form of the loop equations for the classical ensembles}
Consider first the Gaussian $\beta$-ensemble ME${}_{\beta,N}[e^{-\lambda^2/2}]$ (here the rescaling of the eigenvalues $\lambda \mapsto  \lambda/\sqrt{2}$ is for convenience;
recall the text above (\ref{0.8})).
With $J_n$ as in (\ref{2.6}) the $n$-th loop equation is \cite{BEMN10,BMS11,MMPS12,WF14}
\begin{align}
0&=\left[(\kappa-1)\frac{\partial}{\partial x_1}-x_1\right]\overline{W}_n(x_1,J_n)+  N \chi_{n=1} \nonumber
\\&\quad+\chi_{n\neq1}\sum_{k=2}^n\frac{\partial}{\partial x_k}\left\{\frac{\overline{W}_{n-1}(x_1,\ldots,\hat{x}_k,\ldots,x_n)-\overline{W}_{n-1}(J_n)}{x_1-x_k}\right\}\nonumber
\\&\quad+\kappa\left[\overline{W}_{n+1}(x_1,x_1,J_n)+\sum_{J\subseteq J_n}\overline{W}_{|J|+1}(x_1,J)\overline{W}_{n-|J|}(x_1,J_n\setminus J)\right].\label{3.10}
\end{align}
Here the notation $\hat{x}_k$ indicates that the variable $x_k$ is not present in the argument, and thus $\overline{W}_{n-1}(x_1,\ldots,\hat{x}_k,\ldots,x_n) =
\overline{W}_{n-1}( \{x_j\}_{j=1}^n \setminus \{x_k \})$.

Consider next the Laguerre  $\beta$-ensemble ME${}_{\beta,N}[x^{\alpha_1} e^{-x} \chi_{x > 0}]$.
 The $n$-th loop equation is \cite[Eq.~(3.9)]{FRW17}
\begin{align}
\label{eq:loop_laguerre}
0&=\left[(\kappa-1)\frac{\partial}{\partial x_1}+\left(\frac{\alpha_1}{x_1}-1\right)\right]\overline{W}_n(x_1,J_n)+\chi_{n=1}\frac{N}{x_1}\nonumber
\\&\quad+\chi_{n\neq1}\sum_{k=2}^n\frac{\partial}{\partial x_k}\left\{\frac{\overline{W}_{n-1}(x_1,\ldots,\hat{x}_k,\ldots,x_n)-\overline{W}_{n-1}(J_n)}{x_1-x_k}+\frac{1}{x_1}\overline{W}_{n-1}(J_n)\right\}\nonumber
\\&\quad+\kappa\left[\overline{W}_{n+1}(x_1,x_1,J_n)+\sum_{J\subseteq J_n}\overline{W}_{|J|+1}(x_1,J)\overline{W}_{n-|J|}(x_1,J_n\setminus J)\right].
\end{align}

Finally, consider the Jacobi $\beta$-ensemble ME${}_{\beta,N}[x^{\alpha_1} (1 - x)^{\alpha_2} \chi_{0 < x < 1}]$.  The $n$-th loop equation is \cite[Eq.~(4.6)]{FRW17}
\begin{align}
0&=\left((\kappa-1)\frac{\partial}{\partial x_1}+\left(\frac{\alpha_1}{x_1}-\frac{\alpha_2}{1-x_1}\right)\right)\overline{W}_n(x_1,J_n)-\frac{n-1}{x_1(1-x_1)}\overline{W}_{n-1}(J_n)\nonumber
\\&\quad+\frac{\chi_{n=1}}{x_1(1-x_1)}\left[(\alpha_1+\alpha_2+1)N+\kappa N(N-1)\right]-\frac{\chi_{n\neq1}}{x_1(1-x_1)}\sum_{k=2}^nx_k\frac{\partial}{\partial x_k}\overline{W}_{n-1}(J_n)\nonumber
\\&\quad+\chi_{n\neq1}\sum_{k=2}^n\frac{\partial}{\partial x_k}\left\{\frac{\overline{W}_{n-1}(x_1,\ldots,\hat{x}_k,\ldots,x_n)-\overline{W}_{n-1}(J_n)}{x_1-x_k}+\frac{1}{x_1}\overline{W}_{n-1}(J_n)\right\}\nonumber
\\&\quad+\kappa\left[\overline{W}_{n+1}(x_1,x_1,J_n)+\sum_{J\subseteq J_n}\overline{W}_{|J|+1}(x_1,J)\overline{W}_{n-|J|}(x_1,J_n\setminus J)\right].\label{4.5}
\end{align}

\section{Solving the loop equations at low order with $\beta = 2 \alpha/N$ --- the Gaussian $\beta$-ensemble}\label{S3}
Our interest is in the scaling of $\beta$ proportional to the reciprocal of $N$, as specified
by (\ref{0.8}). As a modification of (\ref{2.7a}), we make the ansatz for the large $N$ expansion of 
$\overline{W}_n$ to have the form
\begin{equation}\label{3.1}
\overline{W}_n(x_1,\dots, x_n) = N \sum_{l=0}^\infty {W_n^l(x_1,\dots,x_n) \over N^{l}}.
\end{equation}
Note that this $N$ dependence is consistent with both (\ref{0.4}) and the factor of $1/N$ in
(\ref{0.9+}). 
We will consider each of the three  ensembles separately, beginning with the
Gaussian $\beta$-ensemble.

Consider (\ref{3.10}) with $n=1$. Upon the substitution (\ref{3.1}), by equating terms 
${\rm O}(N)$ we read off the equation for $W_1^0 = W_1^{0,\rm G}$
\begin{equation}\label{3.2}
\Big ( - {d \over d x }  - x \Big )  W_1^{0,\rm G}(x) + 1 + \alpha (  W_1^{0,\rm G}(x) )^2 = 0.
\end{equation}
Equating terms O$(1)$ gives an equation relating $W_1^{0,\rm G}(x), W_1^{1,\rm G}(x), W_2^{0,\rm G}(x,x)$,
\begin{equation}\label{3.2a}
\alpha {d \over d x }   W_1^{0,\rm G}(x)  + \Big (-     {d \over d x }  - x \Big )   W_1^{1,\rm G}(x) +
2 \alpha  W_1^{0,\rm G}(x)    W_1^{1,\rm G}(x)  + \alpha W_2^{0,\rm G}(x,x) = 0.
\end{equation}
For the appearance of  Riccati equations specifying the Stieltjes transform $W_1^0$ of other random matrix models in the context of loop equations, see \cite{EM09}.

We next consider (\ref{3.10}) with $n=2$.  Equating terms ${\rm O}(N)$ gives
an equation relating  $W_2^{0,\rm G}(x_1,x_2)$ to
$W_i^{0,\rm G}(x_i)$ ($i=1,2$). Thus
\begin{equation}\label{3.2b}
\Big (-     {\partial \over \partial x_1 }  - x_1 \Big )   W_2^{0, \rm G}(x_1,x_2)   +
{\partial \over \partial x_2} \bigg \{ {W_1^{0,\rm G}(x_1) - W_1^{0,\rm G}(x_2)   \over x_1 - x_2 } \bigg \} +
2 \alpha W_1^{0,\rm G}(x_1)    W_2^{0,\rm G}(x_1,x_2) = 0.
\end{equation}
Note that with $W_1^{0,\rm G}(x)$ specified by (\ref{3.2}),  (\ref{3.2b}) then allows us to specify
$W_2^{0,\rm G}(x_1,x_2)$.  In relation to $W_2^{0,\rm G}(x,x)$ appearing in (\ref{3.2a}),
we can first take the limit $x_1 \to x_2 = x$ in 
(\ref{3.2b}) to deduce
\begin{equation}\label{3.2b+}
\Big (-  {1 \over 2}    {d \over d x }  - x \Big )   W_2^{0,\rm G}(x,x)   +   {1 \over 2}
{d^2 \over d x^2}  W_1^{0,\rm G}(x)  +
2 \alpha  W_1^{0,\rm G}(x)    W_2^{0,\rm G}(x,x) = 0,
\end{equation}
where in the derivation use has been made of the symmetry $W_2^{0,\rm G}(x_1,x_2) = W_2^{0, \rm G}(x_2,x_1)$.
With  $W_2^{0,\rm G}(x,x)$ so now specified (albeit in terms of $W_1^{0, \rm G}(x)$),
substituting in (\ref{3.2a}) then allows for $W_1^{1,\rm G}(x)$ to be specified.

Let us now carry through this program, and in particular, quantify to what extent it
is possible to specify the quantities of interest.
 With regards to $W_1^{0,\rm G}$, the differential equation (\ref{3.2}) was first obtained
in the present context in \cite{ABG12}, having appeared much earlier in the orthogonal
polynomial literature \cite{AW84} where it relates to so-called associated Hermite polynomials
(for a different line of work in the recent random matrix theory literature relating to associated Hermite polynomials, see \cite{GK20}).
It is an example of a Ricatti nonlinear equation, and
as such can be linearised by setting
\begin{equation}\label{u0}
W_1^{0,\rm G}(x) = - {1 \over \alpha} {d \over d x} \log u(x), \qquad u(x) \mathop{\sim}\limits_{|x| \to \infty}
{A_1 \over x^\alpha},
\end{equation}
for some constant $A_1$. This substitution gives the second order linear
equation for $u(x)$,
\begin{equation}\label{u}
u'' + x u' + \alpha u = 0.
\end{equation}
The solution of (\ref{u}) satisfying the asymptotic condition in (\ref{u0}) is
\cite{ABG12} (see also \cite{Wu19})
\begin{equation}\label{u1}
u(x) = A_2 e^{- x^2/4} D_{-\alpha}(ix),
\end{equation}
where $D_{-\alpha}(z)$ is the so-called parabolic cylinder function with integral
representation
\begin{equation}\label{u2}
D_{-\alpha}(z) = {e^{-z^2/4} \over \Gamma(\alpha) }  \int_0^\infty t^{\alpha-1} e^{-zt - t^2/2} \, dt.
\end{equation}
Substituting in (\ref{u0}) it follows
\begin{equation}\label{u3}
W_1^{0,\rm G}(x) = {x \over 2 \alpha} - {1 \over \alpha} {d \over dx} \log D_{-\alpha}(ix).
\end{equation}

According to (\ref{WG2}) at leading order in $N$, $W_1^{0,\rm G}(x)$ is the
Stieltjes transform of $\rho_{(1),0}^{\rm G}(x;c)$.
The inversion formula (\ref{u4}), with $W_1^{0, \rm G}(x)$ given by (\ref{u3}),
implies \cite{ABG12} the explicit form  of the density
(\ref{0.8+}), or equivalently, upon recalling (\ref{u2})
\begin{equation}\label{u7}
\rho_{(1),0}^{\rm G}(x;\alpha) =
{1 \over \sqrt{2 \pi} \Gamma(1 + \alpha)}
{1 \over | D_{-\alpha}(ix) |^2}.
\end{equation}

\begin{remark}\label{R31}
Suppose in (\ref{3.2}) we scale $x \mapsto \sqrt{\alpha} y$ and $W_1^{0,\rm G}(x) \mapsto {1 \over \sqrt{\alpha}} W_1^{0, \rm G}(y)$. Then for large $\alpha$ (\ref{3.2}) reduces to the quadratic equation
\begin{equation}\label{u8}
- y W_1^{0, \rm G}(y) + 1 + (W_1^{0, \rm G} (y) )^2 = 0,
\end{equation}
with solution obeying $W_1^{0, \rm G}(y) \sim 1/y$ as $y \to \infty$
$$
W_1^{0, \rm G}(y) = {y - (y^2 - 4)^{1/2} \over 2}.
$$
The inversion formula (\ref{u4}) then implies
\begin{equation}
\lim_{\alpha \to \infty} \sqrt{\alpha}
\rho_{(1),0}^{\rm G}(\sqrt{\alpha} y; \alpha) = {1 \over \pi} (4 - y^2)^{1/2}, \qquad |y| < 2,
\end{equation}
which up to scaling is the Wigner semi-circle law
(\ref{0.3}); see also \cite{ABG12} for a discussion of this limit.
\end{remark}

Knowledge of the functional form (\ref{u3}) is not itself of practical use to specify
$W_2^{0,\rm G}(x_1,x_2)$ from (\ref{3.2b}). Instead we view (\ref{3.2}) as specifying the coefficients
$\{ m_{p,0}^{\rm G} \}$ in the expansion about $x = \infty$ of $ W_1^{0, \rm G}(x)$ (\ref{v1}).
 Since the density $\rho_{(1),0}^{\rm G}(x;\alpha) $ is even in $x$, we see that 
\begin{equation}\label{v1a}
m_{p,0}^{\rm G}=0, \qquad {\rm for \:}p \:{\rm odd}.
\end{equation}
Substituting in (\ref{3.2}) gives the recurrence (\ref{v2}).
 An alternative  specification of $\{ m_{2p,0}^{\rm G} \}$ follows by substituting the
known $x \to \infty$ expansion  of $D_{-\alpha}(ix)$ \cite[\S 12.9]{DLMF} in (\ref{u3}). This shows
\begin{align}\label{v2a}
W_1^{0,\rm G}(x) & = {1 \over x} - {1 \over \alpha} {d \over dx} \log \Big ( 1 + \sum_{s=1}^\infty
{(\alpha)_{2s} \over s! (2 x^2)^s } \Big ) \nonumber \\
& = {1 \over x} + {(1+\alpha) \over x^3} + {(3 + 5\alpha + 2 \alpha^2) \over x^5} + {(15 + 32 \alpha + 22 \alpha^2 + 5 \alpha^3) \over x^7} + \cdots
\end{align}
and thus (extending (\ref{v2a}) to include the term O$(1/x^9)$)
\begin{align}\label{v2a+}
m_{2,0}^{\rm G} & = 1 + \alpha  \nonumber  \\
m_{4,0}^{\rm G} & = 3 + 5\alpha + 2 \alpha^2  \nonumber  \\
m_{6,0}^{\rm G} & = 15 + 32 \alpha + 22 \alpha^2 + 5 \alpha^3  \nonumber  \\
m_{8,0}^{\rm G} & = 105 + 260 \alpha + 234 \alpha^2 + 93 \alpha^3 + 14 \alpha^4.
\end{align}

\begin{remark}\label{R1}
1. For even $p \ge 2$  the recurrence (\ref{v2}) can be rewritten as
\begin{equation}\label{v2b}
m_{p+2,0}^{\rm G} = (p+1+2 \alpha ) m_{p,0}^{\rm G} + \alpha \sum_{s=1}^{p/2 - 1} m_{p - 2s,0}^{\rm G} m_{2s,0}^{\rm G}.
\end{equation}
Indeed for even $p\geq 2$ we can rewrite \eqref{v2} as
\begin{equation}
    \begin{split}
        m_{p+2,0}^{\rm G} &= (p+1) m_{p,0}^{\rm G} + \alpha \sum_{s=0}^{p/2} m_{p-2s,0}^{\rm G} m_{2s,0}^{\rm G} \\ 
        &=(p+1) m_{p,0}^{\rm G} + 2\alpha m_{p,0}^{\rm G}m_{0,0}^{\rm G} + \alpha \sum_{s=1}^{p/2-1} m_{p-2s,0}^{\rm G} m_{2s,0}^{\rm G}\\ &=(p+1+2 \alpha ) m_{p,0}^{\rm G} + \alpha \sum_{s=1}^{p/2 - 1} m_{p - 2s,0}^{\rm G} m_{2s,0}^{\rm G},
    \end{split}
\end{equation}
where in the last equality we used that $m_{0,0}^{\rm G} = 1$.
According to  (\ref{v2a+}) $m_{2,0}^{\rm G} = (1 + \alpha)$, so it follows from (\ref{v2b}) that
\begin{equation}\label{v2c}
m_{2n,0}^{\rm G} = (1 + \alpha) \tilde{m}_{n,0}^{\rm G}, \qquad n \ge 1,
\end{equation}
where $ \tilde{m}_{n,0}^{\rm G}$ is a polynomial in $\alpha$ of degree $n-1$ satisfying the recurrence
\begin{equation}\label{v2d}
\tilde{m}_{n+1,0}^{\rm G} = (2n +1 + 2 \alpha) \tilde{m}_{n,0}^{\rm G} + \alpha ( 1 + \alpha) \sum_{s=1}^{n - 1} \tilde{m}_{n-s,0}^{\rm G} \tilde{m}_{s,0}^{\rm G}, \qquad \tilde{m}_{1,0}^{\rm G} = 1, \: \: (n\in \mathbb{N}).
\end{equation}
\color{black}
For combinatorial interpretations, see \cite{Dr09}.

2. As is well known in the theory of the Selberg integral (see \cite[\S 4.1]{Fo10}) the PDF (\ref{0.1}) specifying the Gaussian
$\beta$-ensemble is well defined for $\beta > - 2/N$, implying that the scaling (\ref{0.8}) is well defined for $\alpha > -1$ as
stated; see also \cite{AG13,AB19}. In particular this implies all the moments are non-negative for $\alpha > -1$.
From point 1.~above, we see that exactly at $\alpha = - 1$ all the moments $ \tilde{m}_{2p,0}^{\rm G}$ vanish for $p \ge 1$.

3. For the ensemble ME${}_{\beta,N}(e^{-x^2/2})$ the moments $m_{2k}^{(G)}$ are polynomials in $N$ of
degree $(k+1)$. We know from \cite{DE06,MMPS12,WF14} the explicit forms
\begin{align}\label{v2e}
m_{2}^{(G)} & = \kappa \Big ( N^2 + N ( - 1 + \kappa^{-1})  \Big ) \nonumber \\
m_{4}^{(G)} & =  \kappa^2 \Big ( 2 N^3 + 5 N^2 ( - 1 + \kappa^{-1}) + N (3 - 5 \kappa^{-1} + 3 \kappa^{-2})  \Big ) \nonumber \\
m_{6}^{(G)} & =  \kappa^3 \Big ( 5 N^4 + 22 N^3  ( - 1 + \kappa^{-1})  + N^2 (32 - 54 \kappa^{-1} + 32 \kappa^{-2})  \nonumber \\
& \quad + N (-15 + 32 \kappa^{-1} - 32 \kappa^{-2} + 15 \kappa^{-3})   \Big ) ,
\end{align}
where $\kappa := \beta/2$; in fact \cite{WF14} gives the explicit form of all moments up to and including
$m_{20}^{(G)}$. It follows from (\ref{v2e}) that
\begin{align}\label{v2f}
m_{2}^{(G)} \Big |_{\kappa = \alpha/N} & =  (1 + \alpha) N - \alpha  \nonumber \\
m_{4}^{(G)}  \Big |_{\kappa = \alpha/N}  & = (3 + 5\alpha + \alpha^2) N - 5 \alpha (1 + \alpha) + {3 \alpha^2 \over N}  \nonumber \\
m_{6}^{(G)}  \Big |_{\kappa = \alpha/N}  & = (15 + 32 \alpha + 22 \alpha^2 + 5 \alpha^3) N - 2 \alpha (16 + 27 \alpha + 11 \alpha^2) + {32 \alpha^2 (1 + \alpha) \over N} - {15 \alpha^3 \over N^2} .
\end{align}
We see that the polynomials in $\alpha$ multiplied by $N$ in these expansions agree with the leading moments in the
scaling limit with $\beta$ specified by (\ref{0.8}) as displayed in (\ref{v2a+}).

\end{remark}

To see the utility of (\ref{v1}) in relation to the equation (\ref{3.2b}) relating $W_2^{0, \rm G}(x_1,x_2)$ to $ W_1^{0, \rm G}(x)$, analogous to
(\ref{v1}) introduce the coefficients $\{ \mu_{(p,q),0}^{\rm G} \}$ in the expansion about $x_1, x_2 = \infty$ (\ref{v3}).
 We remark that the reasoning behind the formula in (\ref{v3}) expressing $ \mu_{(p,q),0}^{\rm G} $ in terms of the covariance is to first note from
 (\ref{2.3}) with $n=2$, and the second equation in (\ref{2.4}), that 
 \begin{equation}\label{v5}
 \overline{W}_2^{\rm G}(x_1,x_2) =  \Big \langle \Big (A(x_1) - \langle A(x_1) \rangle \Big )    
 \Big (A(x_2) - \langle A(x_2) \rangle \Big )  \Big \rangle^{\rm G} =:
 {\rm Cov} \, (A(x_1), A(x_2) )^{\rm G}
 \end{equation}
 where $A(x) = \sum_{j=1}^N 1/ (x - \lambda_j)$. Expanding  about  $x_1, x_2 = \infty$  and taking
 the limit $N \to \infty$ with $\beta$ specified by (\ref{0.8}) gives  (\ref{v3}).
 
 The definition in (\ref{v3}) implies the symmetry property
\begin{equation}\label{v4a}
  \mu_{(p,q),0}^{\rm G}  =  \mu_{(q,p),0}^{\rm G}.
  \end{equation}
  It is also immediate that
\begin{equation}\label{v4b}   
  \mu_{(0,q),0}^{\rm G}  =  \mu_{(p,0),0}^{\rm G} = 0.
 \end{equation}
 In addition, the symmetry of the PDF (\ref{0.1}) under the mapping $\lambda_l \mapsto - \lambda_l$ ($l=1,\dots,N$)
 implies
  \begin{equation}\label{vc} 
   \mu_{(p,q),0}^{\rm G}  =  0, \qquad {\rm for \:}p,q{\: \rm of \, different \, parity}.
 \end{equation}     
 
 We substitute both (\ref{v3}) and expansion of (\ref{v1}) in (\ref{3.2b}).
 After straightforward manipulation, this shows
 \begin{multline}\label{vd}
 {1 \over x_1^2} \sum_{p,q=0}^\infty (p+1) {\mu_{(p,q),0}^{\rm G} \over x_1^p x_2^q} - \sum_{p,q=0}^\infty {\mu_{(p,q),0}^{\rm G} \over x_1^p x_2^q}
 + {1 \over x_1 x_2} \sum_{p=0}^\infty m_{p,0}^{\rm G} \sum_{s=0}^p { (s + 1) \over  x_2^{s} x_1^{p-s}   }\\
 + {2 \alpha \over x_1^2}  \sum_{q=0}^\infty {1 \over x_2^q} \sum_{k=0}^\infty {1 \over x_1^k}
 \sum_{s=0}^k m_{s,0}^{\rm G}  \,  \mu_{(k-s,q),0}^{\rm G} = 0.
 \end{multline}
 Taking into consideration the vanishing properties (\ref{v1a}) and (\ref{v4b}), we can further manipulate
 (\ref{vd}) to read
 \begin{multline}\label{ve}
 \sum_{p=3, q=1}^\infty (p-1) {\mu_{(p-2,q),0}^{\rm G} \over x_1^{p-1} x_2^{q-1}} - \sum_{p,q=0}^\infty {\mu_{(p,q),0}^{\rm G} \over x_1^{p-1} x_2^{q-1}}
 + \sum_{r=0}^\infty m_{2 r,0}^{\rm G} \sum_{s=0}^{2r} { (s + 1) \over  x_2^{s} x_1^{2 r-s}   }\\
 +  2 \alpha \sum_{q=1}^\infty {1 \over x_2^{q-1}} \sum_{k=2}^\infty {1 \over x_1^{k-1}}
 \sum_{s=0}^{\lfloor k/2 - 1 \rfloor} m_{2s,0}^{\rm G}   \, \mu_{(k-2-2s,q),0}^{\rm G} = 0.
 \end{multline}
 Equating coefficients of $(x_1^{-p+1} x_2^{-q + 1})$ throughout gives
 the recurrence (\ref{vf}).

  \begin{corollary}\label{C1}
  Let $\{m_{p,0}^{\rm G} \}$ be specified by (\ref{v1a}) and (\ref{v2}). For $q \in \mathbb Z^+$ we have
  \begin{align}\label{x1+}
  {\mu}_{(1,q),0}^{\rm G} & = q m_{q-1,0}^{\rm G}   \nonumber \\
   {\mu}_{(2,q),0}^{\rm G} & = q m_{q,0}^{\rm G}   \nonumber \\
   \mu_{(3,q),0}^{\rm G} & =   2(1+\alpha) q m_{q-1,0}^{\rm G} + q m_{q+1,0}^{\rm G}  \nonumber \\
   \mu_{(4,q),0}^{\rm G} & = (3 + 2\alpha) q m_{q,0}^{\rm G} + q m_{q+2,0}^{\rm G}.
   \end{align}
   \end{corollary}
   
   \begin{remark}\label{RC1}
   1.~The symmetry (\ref{v4a}) is not apparent in (\ref{vf}), and thus not in (\ref{x1+}) either. Nonetheless, on a case-by-case basis,
   the evaluations (\ref{x1}) can be checked to be consistent with (\ref{v4a}). As an example, for
   $\mu_{(2,4),0}^{\rm G} =  \mu_{(4,2),0}^{\rm G}$, the  equality of the corresponding expressions in (\ref{x1+}) requires
   $m_{4,0}^{\rm G} = (3 + 2 \alpha) m_{2,0}^{\rm G}$ which from (\ref{v2a+}) is seen to hold true. \\
   2.~The covariances $\{ \mu_{(p,q),0}^{\rm G}\}$ have been studied in a recent work of Spohn \cite{Sp20}, where they were specified by a certain matrix equation with entries permitting a recursive evaluation. In fact the entry $(p,q)$ ($0 \le p \le q)$ of the matrix equation can be checked to be equivalent to the recurrence (\ref{vf}).
   \end{remark}

  We now turn our attention to the relation (\ref{3.2b}).
  Setting $x_1 = x_2 = x$ in (\ref{v3}) gives the expansion about $x = \infty$ (\ref{w1}).
 Substituting this and (\ref{v1})  in (\ref{3.2b}) gives the recurrence for $\{\tilde{\mu}_{2p,0}^{\rm G}\}$ 
 \begin{equation}\label{w1+} 
 \tilde{\mu}_{2p+2,0}^{\rm G} = (p+1)  \tilde{\mu}_{2p,0}^{\rm G} + (2p+1) (p+1) m_{2p,0}^{\rm G}  + 2 \alpha \sum_{l=1}^p \tilde{\mu}_{2l,0}^{\rm G} m_{2 (p - l), 0}^{\rm G}, \qquad \tilde{\mu}_{0,0}^{\rm G} = 0,
  \end{equation}  
valid for $p=0,1,2,\dots$. Here  $\{m_{2p}^{\rm G} \}$ are input, having been determined by (\ref{v2}).
The first three non-zero values implied by (\ref{w1}) are
 \begin{align}\label{x1}
  \tilde{\mu}_{2,0}^{\rm G} & = 1   \nonumber \\
   \tilde{\mu}_{4,0}^{\rm G} & =  8 ( \alpha + 1)   \nonumber \\
   \tilde{\mu}_{6,0}^{\rm G} & = 3 ( \alpha + 1) (23 + 16 \alpha).
   \end{align}
   Each can be checked to be consistent with the relationship between $\{  \tilde{\mu}_{2p,0}^{\rm G}  \}$ and
   $\{   \mu_{(p_1,q_1),0}^{\rm G} \}$ as specified in (\ref{w1}).
   
   With knowledge of both $\{ W_1^0(x) \}$ as determined by (\ref{v1}), (\ref{v1a}), (\ref{v2}) and
   $\{ W_2^{0,\rm G }(x,x) \}$ as determined by (\ref{w1}), (\ref{w1+}), introducing the expansion
   (\ref{x3})
  the equation (\ref{3.2a}) can be used to deduce a recurrence specifying $\{ m_{2p,0}^{\rm G} \}$,
  which is (\ref{x4m}) in Proposition \ref{P03} above.
  Iterating shows
   \begin{align}\label{x4}
m_{2,1}^{\rm G} & =  - \alpha   \nonumber \\
 m_{4,1}^{\rm G}  & =  - 5 \alpha ( \alpha + 1)   \nonumber \\
   m_{6,1}^{\rm G} & =  - 2 \alpha (16 + 27 \alpha + 11 \alpha^2),
   \end{align}
   which we see are all in agreement with the term independent of $N$ in the expansions (\ref{v2f}).

   \section{Solving the loop equations at low order with $\beta = 2 \alpha/N$ --- the Laguerre $\beta$-ensemble}\label{S4}
For the Laguerre $\beta$-ensemble ME${}_{\beta,N}(x^{\alpha_1} e^{-x})$, $\alpha_1 > -1$ , let the moments of the spectral density
be denoted $m_j^{ (L)}$. Analogous to (\ref{v2e}), each $m_j^{(L)}$ is a polynomial of degree $j$ in $N$ and $\kappa$
(and also in $\alpha_1$).
A listing of $\{  \tilde{m}_j^{ (L)} \}_{j=1,2,3}$ is given in \cite[Prop.~3.11]{FRW17} (see also \cite{MRW15}), where
 \begin{equation}\label{4.1}
 \tilde{m}_j^{(L)} = (N \kappa)^{-j} m_j^{(L)}.
\end{equation} 
 We read off that
 \begin{align}\label{4.2}
\alpha \tilde{m}_{1}^{(L)} \Big |_{\kappa = \alpha/N} & =  (1 + \alpha + \alpha_1)  - {\alpha \over N}  \nonumber \\
\alpha^2 \tilde{m}_{2}^{(L)}  \Big |_{\kappa = \alpha/N}  & = \Big ( (2 + 3 \alpha_1 + \alpha_1^2) + \alpha (4 + 3 \alpha_1) + 2 \alpha^2 \Big ) 
- {\alpha \over N} \Big ((4 + 3 \alpha_1) + 4 \alpha \Big ) + {\rm O} \Big ( {1 \over N^2} \Big )
  \nonumber \\
\alpha^3 \tilde{m}_{3}^{(L)}  \Big |_{\kappa = \alpha/N}  & = \Big ( (6 +11 \alpha_1 + 6 \alpha_1^2 + \alpha_1^3  ) + \alpha(17 + 21 \alpha_1 + 6 \alpha_1^2 ) +
\alpha^2 (16 + 10 \alpha_1) + 5 \alpha^3 \Big )    \nonumber \\
& \quad - {\alpha \over N} \Big ( (17 + 21 \alpha_1 + 6 \alpha_1^2) + \alpha(33 + 21 \alpha_1) + 16 \alpha^2 \Big )  + {\rm O} \Big ( {1 \over N^2} \Big ).
\end{align}
Note that here, in distinction to the case of fixed $N, \beta$, after taking the scaling limit with $\beta = 2c/N$ setting $\alpha_1=-1$ is now well defined. This is of importance for our application of the final section.

In view of this expansion we hypothesise that the functions $\overline{W}_n$ again exhibit the
$N^{-1}$ expansion (\ref{3.1}).
We begin by
enforcing this expansion  in the loop equation \eqref{eq:loop_laguerre} with $n=1$. Equating terms 
${\rm O}(N)$ we read off the equation for $W_1^0 = W_1^{0,\rm L}$ (the use of the superscript ``L'' is to indicate the Laguerre ensemble in the scaling limit with $\beta = 2c/N$),
\begin{equation}
		\label{eq:Lag_N_order}
		-{d \over d x}W_1^{0,\rm L}(x) +\left(\frac{\alpha_1}{x} -1\right)W_1^{0,\rm L}(x) + \frac{1}{x} + \alpha \left(W_1^{0,\rm L}(x)\right)^2 = 0.
	\end{equation}
Equating terms O$(1)$ gives an equation relating $W_1^{0,\rm L}(x), W_1^{1,\rm L}(x), W_2^{0,\rm L}(x,x)$,
\begin{equation}
		\label{eq:Lag_1_order}
		\alpha {d \over d x}W_1^{0,\rm L}(x) - {d \over d x}W_1^{1,\rm L}(x) +  \left(\frac{\alpha_1}{x} -1\right)W_1^{1,\rm L}(x) + \alpha W_2^{0,\rm L}(x,x) + 2 \alpha W_1^{0,\rm L}(x)W_1^{1,\rm L}(x) = 0.
	\end{equation}

We next consider the substitution of (\ref{3.1}) in \eqref{eq:loop_laguerre} with $n=2$.  Equating terms ${\rm O}(N)$ gives
an equation relating  $W_2^{0,\rm L}(x_1,x_2)$ to
$W_i^{0,\rm L}(x_i)$ ($i=1,2$). Thus
\begin{equation}
	\label{eq:Lag_2_N_order}
	\begin{split}
			-{\partial \over \partial x_1}W_2^{0,\rm L}&(x_1,x_2) + \left(\frac{\alpha_1}{x_1} -1\right)W_2^{0,\rm L}(x_1,x_2) \\
			& + {\partial \over \partial x_2}\left\{ \frac{W_1^{0,\rm L}(x_1) - W_1^{0,\rm L}(x_2)}{x_1-x_2} + \frac{W_1^{0,\rm L}(x_2)}{x_1}\right\} + 2\alpha W_2^{0,\rm L}(x_1,x_2)W_1^{0,\rm L}(x_1) = 0\, .
	\end{split}
\end{equation}

We notice that the differential equation \eqref{eq:Lag_N_order} --- a particular Ricatti equation --- can be solved explicitly. This was studied by Allez and collaborators in \cite{ABMV13}, and also appeared earlier in the orthogonal polynomial literature in the context of associated Laguerre polynomials \cite{Le93}. Analogous to (\ref{u0}), the substitution
\begin{equation}\label{u0L}
W_1^{0,\rm L}(x) = - {1 \over \alpha} {d \over d x} \log u(x), \qquad u(x) \mathop{\sim}\limits_{|x| \to \infty}
{B_1 \over x^\alpha},
\end{equation}
for some constant $B_1$, gives rise to the second order differential 
equation,
\begin{equation}\label{uL}
u'' + \Big ( 1 - {\alpha_1 \over x} \Big )  u' + {\alpha \over x}  u = 0.
\end{equation}
The required solution is given by
\cite[Eq.~(3.41), $\alpha \mapsto - 2 \alpha_1$, $\zeta = \alpha + \alpha_1/2$, $\mu = (1 + \alpha_1)/2$]{ABMV13}
\begin{equation}\label{uLa}
u(x) = B_2 e^{-x/2} x^{-\alpha_1/2} W_{-\alpha-\alpha_1/2,(1 + \alpha_1)/2}(-x),
\end{equation}
for some constant $B_2$, where $W_{\zeta,\mu}(z)$ denotes the Whittaker function. Substituting in (\ref{u0L}) shows
\begin{equation}\label{uLb}
W_1^{0,\rm L}(x) = {1 \over 2 \alpha} + {\alpha_1 \over 2 \alpha x} - {1 \over \alpha} {d \over dx }
\log W_{-\alpha - \alpha_1/2,(1 + \alpha_1)/2}(-x),
\end{equation}
and this, upon substituting the known large $x$ form of the Whittaker function \cite[\S 13.19]{DLMF} implies
\begin{align}\label{uLc}
W_1^{0,\rm L}(x) & = {1 \over x} - {1 \over \alpha}
{d \over d x} \log \Big ( 1 +
\sum_{s=1}^\infty {(\alpha)_s (1 +\alpha_1 + \alpha)_s \over s!} {1 \over x^s} \Big )  \nonumber \\
& = {1 \over x} + {(1 + \alpha_1 + \alpha) \over x^2} +
{(1 + \alpha_1 + \alpha)(2 \alpha + 2 + \alpha_1) \over x^3} \nonumber\\
& \qquad + {(1 + \alpha_1 + \alpha) (6+11 \alpha + 5 \alpha^2 + 5 (1 + \alpha) \alpha_1 + \alpha_1^2) \over x^4} + \cdots
\end{align}

Analogous to (\ref{v1}) $W_1^{0, \rm L}(x)$ is the
moment generating function of the corresponding Laguerre
$\alpha$-ensemble density,
\begin{equation}\label{v1L}
 W_1^{0, \rm L}(x) = {1 \over x} \sum_{p=0}^\infty {m_{p,0}^{\rm L} \over x^p }, \qquad m_{p,0}^{\rm L} = \int_{0}^\infty x^p  \rho_{(1),0}^{\rm L}(x;\alpha_1,\alpha)  \, dx.
\end{equation}
We thus read off from (\ref{uLc}) that
\begin{align}\label{Lk1}
    m_{1,0}^{\rm L} & = (1 + \alpha_1 + \alpha) \nonumber \\
m_{2,0}^{\rm L} & = (1 + \alpha_1 + \alpha)
(2 \alpha + 2 + \alpha_1) \nonumber \\
m_{3,0}^{\rm L} & = (1 + \alpha_1 + \alpha)
(6+11 \alpha + 5 \alpha^2 + 5 (1 + \alpha) \alpha_1 + \alpha_1^2).
\end{align}
These are all in agreement with the leading terms (in $N$) on the RHS of (\ref{4.2}). Furthermore, by substituting the expansion (\ref{v1L}) in the differential equation 
\eqref{eq:Lag_N_order} we see $\{m_{p,0}^{\rm L}\}$ satisfies the recurrence
\begin{equation}\label{v2Lm}
m_{p+1,0}^{\rm L} = (p+1 + \alpha_1 + \alpha) m_{p,0}^{\rm L} + \alpha \sum_{s=0}^{p-1} m_{s,0}^{\rm L} m_{p-s,0}^{\rm L}, \qquad{m}_{0,0}^{\rm L} = 1.
\end{equation} 
An immediate corollary is that $m_{p,0}^{\rm L}$ is a polynomial of degree $p$ in both $\alpha_1, \alpha$, as seen in the tabulation (\ref{Lk1}) for the low order cases. Moreover, analogous to point 1.~of Remark \ref{R1}, writing
$$
m_{p,0}^{\rm L} = (1 + \alpha_1 + \alpha)
\tilde{m}_{p,0}^{\rm L}, \qquad p \ge 1
$$
we see that $\tilde{m}_{p,0}^{\rm L}$ is a polynomial of degree $p-1$ in both $\alpha_1, \alpha$, satisfying
the recurrence
\begin{equation}\label{v2L}
\tilde{m}_{p+1,0}^{\rm L} = (p+1 + \alpha_1 + 2 \alpha) 
\tilde{m}_{p,0}^{\rm L} + \alpha (1 + \alpha_1 + \alpha)
\sum_{s=1}^{p-1}
\tilde{m}_{s,0}^{\rm L} \tilde{m}_{p-s,0}^{\rm L}, \qquad \tilde{m}_{1,0}^{\rm L} = 1.
\end{equation} 

The density $\rho_{(1),0}^{\rm L}(x;\alpha_1,\alpha)$ in (\ref{v1L}) can be deduced from knowledge of $W_1^{0, \rm L}$ as specified by (\ref{uLb}), together with the analogue of the inversion formula (\ref{u4}).
One finds  \cite[Eq.~(3.49), $\lambda \mapsto  2x$, $\zeta = \alpha + \alpha_1/2$, $\mu = (1 + \alpha_1)/2$]{ABMV13}
\begin{equation}\label{uLd}
\rho_{(1),0}^{\rm L}(x;\alpha_1,\alpha) =
{1 \over \Gamma(\alpha + 1) \Gamma(\alpha + \alpha_1 + 1)}
{1 \over | W_{-\alpha - \alpha_1/2,(1+\alpha_1)/2}(-x)|^2},
\end{equation}
supported on $x > 0$.

In relation to \eqref{eq:Lag_2_N_order}, introduce the Laguerre analogue of (\ref{v3})
\begin{equation}\label{v3L}
W_2^{0, \rm L}(x_1,x_2) = {1 \over x_1 x_2}  \sum_{p,q=1}^\infty   {   \mu_{(p,q),0}^{\rm L}  \over x_1^p x_2^q},
\qquad
 \mu_{(p,q),0}^{\rm L}  =    \lim_{N \to \infty \atop \beta = 2 \alpha/ N} {\rm Cov} \, \Big ( \sum_{i=1}^N x_i^p, \sum_{i=1}^N x_i^q \Big )^{\rm L}.
 \end{equation}
Proceeding as in the derivation of (\ref{vf}) shows
\begin{align}\label{v4L}
\mu_{(p+1,q),0}^{\rm L} & = (p+1+\alpha_1) \mu_{(p,q),0}^{\rm L} + q m_{p+q,0}^{\rm L} +
2 \alpha \sum_{s=0}^{p-1} m_{s,0}^{\rm L} \mu_{(p-s,q),0}
\nonumber \\
& = (p+1+\alpha_1 +2 \alpha) \mu_{(p,q),0}^{\rm L} + q m_{p+q,0}^{\rm L} +
2 \alpha \sum_{s=1}^{p-1} m_{s,0}^{\rm L} \mu_{(p-s,q),0}.
\end{align}

\begin{corollary}\label{C2}
  Let $\{m_{p,0}^{\rm L} \}$ be specified by (\ref{v1L}) and (\ref{v2L}). For $q \in \mathbb Z^+$ we have
  \begin{align}\label{d1+}
  {\mu}_{(1,q),0}^{\rm L} & = q m_{q,0}^{\rm L}   \nonumber \\
   {\mu}_{(2,q),0}^{\rm L} & = (2 + \alpha_1 + 2 \alpha) q m_{q,0}^{\rm L}  + q m_{q+1,0}^{\rm L} \nonumber \\
   \mu_{(3,q),0}^{\rm L} & =   (3 + \alpha_1 +2 \alpha) \mu_{(2,q)}^{\rm L} +
   2 \alpha (1 + \alpha_1 +2 \alpha)q m_{q,0}^{\rm L} + q m_{q+2,0}^{\rm L}.
   \end{align}
   \end{corollary}
 
 \begin{remark} As with $\{ \mu_{(p,q),0}^{\rm G} \}$, the recurrence (\ref{v4L}) is not symmetric upon the interchange $p \leftrightarrow q$, yet from the definition $\mu_{(p,q),0}^{\rm L}$ has this symmetry. As observed in Remark \ref{RC1} in the Gaussian case, on a case-by-case basis this symmetry can be checked from the explicit forms, in particular those in Corollary \ref{C2} combined with the tabulation (\ref{Lk1}).
 \end{remark}

The equation \eqref{eq:Lag_1_order}
for $W_1^{1, \rm L}$ requires knowledge of $W_2^{0,\rm L}(x,x)$.
 In regards to this quantity, letting $x_1\to x_2=x$ in \eqref{eq:Lag_2_N_order} shows
\begin{equation}\label{WxxL}
\begin{split}
    	-{d \over d x} W_2^{0,\rm L}(x,x) + 2\left(\frac{\alpha_1}{x} -1\right)W_2^{0,\rm L}(x,x) +  {d^2 \over d x^2}W_1^{0,\rm L}(x)&+ {2 \over x} {d \over d x} W_1^{0,\rm L}(x) \\ &+ 4\alpha W_2^{0, \rm L}(x,x)W_1^{0,\rm L}(x) = 0.
\end{split}
\end{equation} 
Introducing the expansion about $x = \infty$
\begin{equation}\label{w1L}
  W_2^{0,\rm L}(x,x) = {1 \over x^2} \sum_{p=1}^\infty {\tilde{\mu}_{p,0}^{\rm L} \over x^{p}}, \qquad
  \tilde{\mu}_{2p,0}^{\rm L}  = \sum_{p_1 + q_1 = p} \mu_{(p_1,q_1),0}^{\rm L}
 \end{equation}  
(cf.~(\ref{w1})), as well as the analogous expansion for $W_1^{0,\rm L}$ from (\ref{v1L}), reduces (\ref{WxxL}) to the recurrence
\begin{equation}\label{w2L}
 \tilde{\mu}_{p+1,0}^{\rm L} =
{1 \over 2} (p+2 + 2 \alpha_1 + 4 \alpha)
\tilde{\mu}_{p,0}^{\rm L} + {p (p+1) \over 2}
m_{p,0}^{\rm L} + 2 \alpha \sum_{s=1}^{p-1} 
\tilde{\mu}_{s,0}^{\rm L} m_{p-s,0}^{\rm L}.
\end{equation}
As done in relation to (\ref{w1+}), the implied evaluations for members of $\{ \tilde{\mu}_{p,0}^{\rm L}\}$ can be checked, for small $p$ at least, to be consistent with the relationship to $\{   \mu_{(p_1,q_1),0}^{\rm G} \}$,
and thus the tabulation (\ref{d1+}),
as required by the second equation in (\ref{w1L}).

With $\{m_{p,0}^{\rm L} \}$ determined by (\ref{uLc}) or
the recurrence (\ref{v2L}),
and $\{ \tilde{\mu}_{p,0}^{\rm L}\}$ determined by the recurrence (\ref{w2L}), by introducing the expansion
\begin{equation}\label{y1L}
  W_1^{1,\rm L}(x) = {1 \over x} \sum_{p=1}^\infty {m_{p,1}^{\rm L} \over x^{p}}
  \end{equation}
  we see from (\ref{eq:Lag_1_order}) that $\{m_{p,1}^{\rm L}\}$ can be determined by the recurrence
\begin{equation}\label{y2L}  
m_{p+1,1}^{\rm L} = - \alpha (p+1) m_{p,0}^{\rm L} + (p+1+\alpha_1 + 2 \alpha) m_{p,1}^{\rm L} + \alpha \tilde{\mu}_{p,0}^{\rm L} + 2 \alpha \sum_{l=1}^{p-1} \tilde{\mu}_{l,0}^{\rm L} m_{p-l,0}^{\rm L},
\end{equation}
valid for $p=1,2,\dots$ with initial condition $m_{0,1}^{\rm L}=0$. In particular, iteration shows
   \begin{align}\label{x4L}
m_{1,1}^{\rm L} & = - \alpha    \nonumber \\
 m_{2,1}^{\rm L}  & = - \alpha  \Big ((4 + 3 \alpha_1) + 4 \alpha \Big )  \nonumber \\
   m_{3,1}^{\rm L} & = - \alpha \Big ( (17 + 21 \alpha_1 + 6 \alpha_1^2) + \alpha(33 + 21 \alpha_1) + 16 \alpha^2 \Big ),
   \end{align}
   which we see are all in agreement with the term independent of $N$ 
   exhibited in the expansions (\ref{4.2}).

  \section{Solving the loop equations at low order with $\beta = 2 \alpha/N$ --- the Jacobi $\beta$-ensemble}\label{S5}
   For the Jacobi $\beta$-ensemble ME${}_{\beta,N}(x^{\alpha_1}(1 -  x)^{\alpha_2})$, let the moments of the spectral density
be denoted $m_j^{(J)}$. In distinction to the Gaussian and Laguerre $\beta$-ensembles, the
moments of the  Jacobi $\beta$-ensemble spectral density are no longer polynomials in $N$ and $\kappa$, but rather 
rational functions. The first two are given explicitly in \cite[App.~B]{MRW15}. From these we deduce
 \begin{align}\label{5.2}
 {1 \over N} m_1^{(J)} & = {\alpha_1 + 1 + \alpha \over \alpha_1 + \alpha_2 + 2 + 2 \alpha}
 - {1 \over N} { \alpha ( \alpha_2 - \alpha_1) \over (\alpha_1 + \alpha_2 + 2 + 2 \alpha)^2} 
 + {\rm O}\Big ( {1 \over N^2} \Big )
 \nonumber \\
 {1 \over N} m_2^{(J)} &  =  { (1 + \alpha + \alpha_1) \Big ( (2 + \alpha_1)(2 + \alpha_1 + \alpha_2)+
 \alpha (7 + 3 \alpha_1 + 2 \alpha_2) +  3 \alpha^2 \Big ) \over
 (2 + 2 \alpha + \alpha_1 + \alpha_2)^2 (3 + 2 \alpha + \alpha_1 + \alpha_2)}   \nonumber \\
 & \quad + {\alpha \over N} {Q_1(\alpha_1, \alpha_2,\alpha) + Q_2(\alpha_1, \alpha_2,\alpha)   \over
 (2 + 2 \alpha + \alpha_1 + \alpha_2)^3 (3 + 2 \alpha + \alpha_1 + \alpha_2)^2} + {\rm O} \Big ( {1 \over N^2} \Big ),
 \end{align}
 where
 $$
 Q_1(\alpha_1, \alpha_2,\alpha)  = - (1 + \alpha + \alpha_1) (2 + 2 \alpha + \alpha_1 + \alpha_2) (3 + 2 \alpha + \alpha_1 + \alpha_2)
 (9 + 7 \alpha + 4 \alpha_1 + 2 \alpha_2),
 $$
 \begin{align*}
  Q_2(\alpha_1, \alpha_2,\alpha)  & = \Big ( (2 + \alpha_1)(2 + \alpha_1 + \alpha_2) + \alpha (7 + 3 \alpha_1 + 2 \alpha_2) +3 \alpha^2  \Big )  \\
  & \quad \times 
  \Big ( 13 + 21 \alpha_1 + 2 \alpha_2 + (6 \alpha_1 - \alpha_2) (\alpha_1 + \alpha_2) + \alpha(23 + 17 \alpha_1 + 3 \alpha_2) + 10 \alpha^2 \Big ).
  \end{align*}


 As in the Gaussian and Laguerre cases, to solve the Jacobi $\beta$-ensemble loop equations \eqref{4.5} in the regime $\beta = 2 \alpha/N$ we will make the ansatz (\ref{3.1}). We see that with $n=1$ the latter is consistent with the form of the expansions (\ref{5.2}).
 Equating the terms of order ${\rm O}(N)$ gives
 the particular Riccati type equation for   $W_1^{0,\rm J}$ ( here we use the superscript ''J'' to indicate the Jacobi ensemble in the scaling limit with $\beta = 2c/N$)
 \begin{equation}
		\label{eq:Jac_N_order}
		-{d \over d x} W_1^{0,\rm J}(x) +\left(\frac{\alpha_1}{x} -\frac{\alpha_2}{1-x}\right)W_1^{0,\rm J}(x) + \frac{1}{x(1-x)}\left(1+\alpha_1+\alpha_2+\alpha \right) + \alpha \left(W_1^{0,\rm J}(x)\right)^2 = 0.
	\end{equation}
	
Being a Ricatti type equation, it is most natural to proceed in the analysis of (\ref{3.2}) and (\ref{eq:Lag_N_order}) and perform	the change of variables
\begin{equation}
\label{eq:psichange}
    W_1^{0, \rm J}(x) = -\frac{\psi'(x)}{\alpha \psi(x)}, \qquad \psi(x) \mathop{\sim}\limits_{|x| \to \infty} {C_1 \over x^\alpha} 
\end{equation}
for some constant $C_1$.
We see that $\psi(x)$ satisfies the second order linear differential equation
\begin{equation}
    x(x-1)\psi''(x)  + \left( \alpha_1 - (\alpha_1 + \alpha_2)x\right)\psi'(x) - \alpha\left( 1+\alpha + \alpha_1 + \alpha_2\right)
    \psi(x) = 0,
\end{equation}
this being a particular hypergeometric differential
equation \cite[\S 15.10]{DLMF}. Due to the condition \eqref{eq:psichange} we have for the general solution 
\begin{equation}
\label{eq:hyp_sol}
    \psi(x) = C_1x^{-\alpha}{}_2F_1\left( \alpha, \alpha + \alpha_1+1, 2\alpha + \alpha_1+\alpha_2 + 2; x^{-1} \right),
\end{equation}
where ${}_2 F_1$ denotes the usual Gauss hypergeometric function.
After some algebraic manipulations, this implies
\begin{align}\label{MPJ1}
W_1^{0, \rm J}(x) & = \frac{1}{x} + \frac{\alpha + \alpha_1+1}{2\alpha + \alpha_1+\alpha_2 + 2}\frac{{}_2F_1\left( \alpha+1, \alpha + \alpha_1+2, 2\alpha + \alpha_1+\alpha_2 + 3; x^{-1} \right)}{ x^2 {}_2F_1\left( \alpha, \alpha + \alpha_1+1, 2\alpha + \alpha_1+\alpha_2 + 2; x^{-1} \right)}\, \nonumber \\
& = -
{{}_2 F_1 (\alpha+1, \alpha + \alpha_1 + 1, 2 \alpha + \alpha_1 + \alpha_2 + 2;1/x) \over x \,
{}_2 F_1 (\alpha, \alpha + \alpha_1 + 1, 2 \alpha + \alpha_1 + \alpha_2 + 2;1/x)}.
\end{align}
This latter form was given recently by Trinh and Trinh \cite{TT20}, using a different set of ideas stemming from the theory of associated Jacobi polynomials \cite{Wi87}, and making no direct use of differential equations.

 Since analogous to (\ref{v1}) and (\ref{v1L})
\begin{equation}\label{JT2}
W_1^{0, \rm J}(x)
= {1 \over x} \sum_{p=0}^\infty {m_{p,0}^{\rm J} \over x^p }, \qquad m_{p,0}^{\rm J} = \int_{0}^1 x^p  \rho_{(1),0}^{\rm J}(x;\alpha_1,\alpha_2,\alpha)  \, dx,
\end{equation}
we can use (\ref{MPJ1}) (with the help of computer algebra) to compute $\{m_{p,0}^{\rm J}\}_{p=1}^\infty$, at least for small $p$. Agreement with the leading order (in $N$) rational functions known from (\ref{5.2}) is found.

It is furthermore the case that substitution of (\ref{JT2}) in (\ref{eq:Jac_N_order}) implies a recurrence for $\{ m_{p,0}^{\rm J}\}_{p=1}^\infty$. Thus we find
\begin{equation}\label{Jr1}
m_{p,0}^{\rm J} = {1 \over p + 1 + \alpha_1 + \alpha_2 + 2 \alpha} \Big ( (1 + \alpha_1  + \alpha) - \alpha_2 \sum_{s=1}^{p-1} m_{s,0}^{\rm J} - \alpha \sum_{s=1}^{p-1} m_{s,0}^{\rm J} m_{p-s,0}^{\rm J} \Big ),
\end{equation}
valid for $p=1,2,\dots$ and subject to the initial condition $m_{0,0}^{\rm J} = 1$. We can verify that iterating for small $p$ ($p=1,2$) reproduces the leading order terms from (\ref{5.2}), and is thus in agreement with (\ref{MPJ1}).

\begin{remark}
1.~Moving the denominator in the RHS of (\ref{Jr1}) to the LHS, replacing $p$ by $p+1$, then subtracting from the form without this latter replacement shows
\begin{equation}\label{Jr1}
m_{p+1,0}^{\rm J} = {1 \over p + 2 + \alpha_1 + \alpha_2 + 2 \alpha} \Big ( (p+ 1 + \alpha_1  ) m_{p,0}^{\rm J}
-  \alpha \sum_{s=1}^{p} m_{s,0}^{\rm J} m_{p+1-s,0}^{\rm J}
+ \alpha \sum_{s=0}^{p} m_{s,0}^{\rm J} m_{p-s,0}^{\rm J} \Big ).
\end{equation}
This recurrence was obtained recently in the work
\cite[Eq.~(15)]{TT20}, which as in the derivation of
(\ref{MPJ1}) in that work uses a different set of ideas. \\
2.~Changing variables $x = (X + 1)/2$ in (\ref{eq:Jac_N_order}) shows
\begin{equation}
		\label{Vd}
		\begin{split}
		    		-{d \over d X} W_1^{0,\rm J^*}(X) +\left(\frac{\alpha_2}{1-X} -\frac{\alpha_2}{1+X}\right)W_1^{0,\rm J^*}(X) + \frac{1}{1 - X^2}&\left(1+\alpha_1+\alpha_2+\alpha \right) \\ & + \alpha \left(W_1^{0,\rm J^*}(X)\right)^2 = 0,
		\end{split}
	\end{equation}
	where
	$$
	W_1^{0,\rm J^*}(X) = \int_{-1}^1 {\rho_{(1),0}^{\rm J^*}(Y;\alpha_1,\alpha_2,\alpha) \over X - Y} \, dY.
	$$
	Here $\rho_{(1),0}^{\rm J^*}$ denotes the density for the Jacobi $\beta$-ensemble with high temperature scaling (\ref{0.8}) relating to the weight
	$(1-X)^{\alpha_1} (1 + X)^{\alpha_2}$ supported on $(-1,1)$. In the case $\alpha_1 = \alpha_2 = a$ (symmetric Jacobi weight) the corresponding moments, as for the Gaussian ensemble, must vanish for $p$ odd. It follows from (\ref{Vd}) that the even moments $\{m_{2p,0}^{\rm J^*} \}$ satisfy the recurrence
	\begin{equation}
	\label{Vd1}
	m_{2p,0}^{\rm J^*} = {1 \over 2p + 2 \alpha + 2 a + 1}
	\Big ( (1+\alpha) - 2 a \sum_{s=1}^{p-1} m_{2s,0}^{\rm J^*} -
	\alpha \sum_{s=1}^{p-1} m_{2s,0}^{\rm J^*}
	m_{2(p-s),0}^{\rm J^*} \Big ),
	\end{equation}
	valid for $p=1,2,\dots$ with initial condition $m_{0,0}^{\rm J^*} = 1$.
	Moments for the symmetric Jacobi  $\beta$-ensembles, with $\beta = 1,2$ or 4 have been the subject of the recent work \cite{FR20}. In fact a number of recent works in random matrix theory have identified recurrences for moments and also distribution functions; see e.g. \cite{Ku19,FK19,CMOS18,ABGS20,RF19,GGR20,GGR20b,FK20a,FK20b}.
	
	Manipulation of (\ref{Vd1}) as in the derivation of (\ref{Jr1}) shows
	\begin{equation}
	\label{JG}
	m_{2(p+1),0}^{\rm J^*} = {1 \over 2p + 2 \alpha + 2 a + 3}
	\Big ( (2p + 1)m_{2p,0}^{\rm J^*} 
	-\alpha \sum_{s=1}^{p} m_{2s,0}^{\rm J^*}m_{2(p+1-s),0}^{\rm J^*}
	 +
	\alpha \sum_{s=0}^{p} m_{2s,0}^{\rm J^*}
	m_{2(p-s),0}^{\rm J^*} \Big ).
	\end{equation}
Scaling $x \mapsto x/\sqrt{2a}$ in the definition of the symmetric moment $m_{2k,0}^{\rm J^*}$	shows
$$
m_{2k,0}^{\rm J^*} = {1 \over (2 a)^{k+1/2}}
\int_{-\sqrt{a}}^{\sqrt{a} }
x^{2k} 
\rho_{(1),0}^{\rm J^*}(x/\sqrt{2a};a,\alpha) \, dx \sim{1 \over (2 a )^{k + 1/2} } m_{2k,0}^{{\rm G}},
$$
where the asymptotic relation follows from the elementary limit $(1 - x^2/2 a)^a \to e^{- x^2}$ as $x \to \infty$. Using this to equate leading order terms in (\ref{JG}) reclaims (\ref{v2}).
\end{remark}

The work \cite[Th.~A.4]{TT20} also contains an explicit formula for the density $\rho_{(1),0}^{\rm J}$. This is derived not from (\ref{eq:psichange}) and an inversion formula analogous to (\ref{u4}), but rather by using theory relating to the asymptotic of associated Jacobi polynomials \cite{IM91} and general relations between tridiagonal matrices and orthogonal polynomials
\cite{Ne79}. With
\begin{align*}
  & U(x) = {\Gamma (\alpha + 1) \Gamma (\alpha_1 + 1) \over \Gamma (1 + \alpha + \alpha_1)} \, 
    {}_2 F_1 (\alpha, -\alpha - \alpha_1 -\alpha_2 - 1,  -\alpha_1 ;x), \\
    & V(x) = {- \pi \alpha \Gamma(\alpha + \alpha_1 +\alpha_2 + 2) \over \sin ( \pi \alpha_1) \Gamma(1 + \alpha + \alpha_2) \Gamma (2 + \alpha_1)}(1 - x)^{1 + \alpha_2}
    x^{1 + \alpha_1} \,
 {}_2 F_1 (1-\alpha, 2+\alpha + \alpha_1 +\alpha_2 ,  2+\alpha_1 ;x),  
\end{align*}
we read off from \cite{TT20} that
\begin{equation}\label{JT3}
\rho_{(1),0}^{\rm J}(x;\alpha_1,\alpha_2,\alpha) =
{\Gamma(\alpha+1) \Gamma(\alpha + \alpha_1 + \alpha_2 + 2) \over \Gamma(\alpha+\alpha_1+1) \Gamma(\alpha 
+\alpha_2 + 1) } 
{x^{\alpha_1} (1 - x)^{\alpha_2} \over 
| U(x) + e^{\pi i \alpha_1} V(x) |^2},
\end{equation}
supported on $0 < x < 1$.

Knowledge of (\ref{eq:psichange}) and (\ref{MPJ1}), together with the inversion formula 
\begin{equation}\label{Jd0}
\rho_{(1),0}^{\rm J}(x;\alpha_1,\alpha_2,\alpha) =
\lim_{\epsilon \to 0^+} {1 \over \pi}
{\rm Im} \, W_1^{0, \rm J}(x - i \epsilon),
\end{equation}
can in fact be used to derive (\ref{JT3}). The starting point is to make use of the connection formula
\cite[\S 15.10(ii)]{DLMF} 
$$
e^{\pi i \alpha} \psi(x) = U(x) + e^{-\pi i \alpha_1} V(x).
$$
Substituting in (\ref{eq:psichange}), then substituting the result in (\ref{Jd0}) shows
\begin{equation}\label{Jd1}
\rho_{(1),0}^{\rm J}(x;\alpha_1,\alpha_2,\alpha) = C_2
{ u'(x) v(x) - v'(x) u(x) \over
|U(x) + e^{-\pi i \alpha_1} V(x)|^2 }
\end{equation}
where
\begin{equation}\label{Jd2}
C_2 = - {1 \over (\alpha_1 + 1)}
{\Gamma(\alpha+1) \Gamma(\alpha + \alpha_1 + \alpha_2 + 2) \over \Gamma(1 + \alpha + \alpha_1) \Gamma(1 + \alpha + \alpha_2)}
\end{equation}
and
$$
u(x) = {}_2 F_1(a,b,c;x), \qquad 
v(x) = x^{1 - c} \, {}_2 F_1(a-c+1,b-c+1,2-c;x)
$$
with
\begin{equation}\label{Jd3}
a = \alpha, \qquad b = - (\alpha + \alpha_1 + \alpha_2 + 1), \qquad c = - \alpha_1.
\end{equation}
Here $u(x), v(x)$ satisfies the same hypergeometric differential equation. We can use this to show
\begin{equation}\label{Jd4}
u'(x) v(x) - v'(x) u(x) = (c-1) x^a (1 - x)^{c - a - b - 1}.
\end{equation}
Substituting (\ref{Jd4}) with parameters given by (\ref{Jd3}) in (\ref{Jd1}) we reclaim (\ref{JT3}).

\begin{remark}
From the relationship between the Jacobi and Laguerre weights we must have
\begin{equation}\label{Jd5}
\lim_{\alpha_2 \to \infty} {1 \over \alpha_2}
\rho_{(1),0}^{\rm J}(x/\alpha_2;\alpha_1,\alpha_2,\alpha)=
\rho_{(1),0}^{\rm L}(x;\alpha_1,\alpha).
\end{equation}
Starting from (\ref{JT3}), and upon making sue of standard asymptotics for the gamma function and the hypergeometric function confluent limit formula
$$
\lim_{b \to \infty} \, {}_2 F_1 (a,b,c;x/b) = {}_1 F_1 (a,c;x)
$$
we see that
\begin{equation}\label{Ld3}
\lim_{\alpha_2 \to \infty} {1 \over \alpha_2}
\rho_{(1),0}^{\rm J}(x/\alpha_2;\alpha_1,\alpha_2,\alpha)=
{\Gamma(\alpha + 1) \over \Gamma(\alpha + \alpha_1 + 1)}
{x^{\alpha_1} e^{-x} \over | \tilde{U}(x) + e^{\pi i \alpha_1} \tilde{V}(x) |^2},
\end{equation}
where
\begin{align*}
\tilde{U}(x) & = {\Gamma(\alpha+1) \Gamma(\alpha_1+1) \over  \Gamma(1 + \alpha + \alpha_1)} \,
{}_1 F_1(\alpha,-\alpha_1;-x) \\
\tilde{V}(x) & = - {\pi \alpha \over\sin(\pi \alpha_1) \Gamma(2 + \alpha_1)}
x^{1 + \alpha_1} e^{-x} \,
{}_1 F_1(1-\alpha,2+\alpha_1;x).
\end{align*}
We see that (\ref{Ld3}) is 
 consistent with (\ref{uLd}) if it is true
 $$
 {1 \over \Gamma(\alpha + 1)}
 | \tilde{U}(x) + e^{\pi i \alpha_1} \tilde{V}(x) | =
 x^{\alpha_1/2} e^{-x/2} |W_{-\alpha-\alpha_1/2,
 (1 + \alpha_1)/2}(-x) |
 $$
 By writing the Whittaker function in terms of the Tricomi hypergeometric function, then writing the latter in terms of the confluent hypergeometric function (see \cite[below Lemma 2.1]{TT19}) this is indeed seen to be valid.
\end{remark}

Coming back to the loop equation for the Jacobi ensemble,
applying \eqref{4.5} with $n=2$ and equating the terms of ${\rm O}(N)$ we get a partial differential  equation for $W_2^{0,J}$ in terms of $W_1^{0,J}$,
	\begin{multline}
	\label{eq:Jac_2_N_order}
		-{\partial \over \partial x_1}W_2^{0,\rm J}(x_1,x_2)  + \left(\frac{\alpha_1}{x_1} -\frac{\alpha_2}{1-x_1}\right)W_2^{0,\rm J}(x_1,x_2)  -\frac{1}{x_1(1-x_1)}\Big (
		1 + x_2 {\partial \over \partial x_2} \Big )
		W_1^{0,\rm J}(x_2)
		\\  + {\partial \over \partial x_2} \left\{ \frac{W_1^{0,\rm J}(x_1) - W_1^{0,\rm J}(x_2)}{x_1-x_2}   + \frac{W_1^{0,\rm J}(x_2)}{x_1}\right\} + 2\alpha W_2^{0,\rm J}(x_1,x_2)W_1^{0,\rm J}(x_1) = 0\, .
\end{multline}
Introducing 
\begin{equation}\label{v3L}
W_2^{0, \rm J}(x_1,x_2) = {1 \over x_1 x_2}  \sum_{p,q=1}^\infty   {   \mu_{(p,q),0}^{\rm J}  \over x_1^p x_2^q},
\qquad
 \mu_{(p,q),0}^{\rm J}  =    \lim_{N \to \infty \atop \beta = 2 \alpha/ N} {\rm Cov} \, \Big ( \sum_{i=1}^N x_i^p, \sum_{i=1}^N x_i^q \Big )^{\rm J}.
 \end{equation}
Proceeding as in the derivation of (\ref{vf}) and (\ref{v4L}) shows
\begin{align}\label{v4J}
\mu_{(p,q),0}^{\rm J} & = 
{1 \over (p+ \alpha_1 + \alpha_2 + 2 \alpha + 1) }
\Big (  q(m_{q,0}^{\rm J} -  m_{p+q,0}^{\rm J} )
-\alpha_2 \sum_{s=1}^{p-1} \mu_{(s,q),0}^{\rm J}
- 2 \alpha \sum_{s=1}^{p-1} m_{s,0}^{\rm J} \mu_{(p-s,q),0}^{\rm J} \Big ).
\end{align}
Note that this is consistent with the requirement that $
\mu_{(0,q),0}^{\rm J} = \mu_{(p,0),0}^{\rm J} = 0$. Beyond this, the simplest case is $p=1$ which gives
\begin{equation}\label{v4Ja}
\mu_{(1,q),0}^{\rm J} = {q \over 
( \alpha_1 + \alpha_2 + 2 \alpha + 2)}
(m_{q,0}^{\rm J} -  m_{p+1,0}^{\rm J} ).
\end{equation}
Thus, for example, making use of knowledge of $m_{1,0}^{\rm J}$ and $m_{2,0}^{\rm J}$ as implied by (\ref{Jr1}), or as can be read off from (\ref{5.2}), we have
\begin{equation}\label{v4Jb}
\mu_{(1,1),0}^{\rm J} = {(1 + \alpha + \alpha_1)
(1 + \alpha + \alpha_2)(2 + \alpha + \alpha_1 + \alpha_2) \over
(2 + 2\alpha + \alpha_1 + \alpha_2)^3(3 + 2\alpha + \alpha_1 + \alpha_2)}.
\end{equation}
We remark that iterating (\ref{v4J}) with the help of computer algebra, we can check the required symmetry
$\mu_{(p,q),0}^{\rm J} = \mu_{(q,p),0}^{\rm J}$ in low order cases.

Finally, we return to the Jacobi ensemble loop equation (\ref{4.5}) in the case $n=1$. With 
$\beta = 2 \alpha/N$ and the ansatz corresponding to (\ref{3.1}), equating the terms of order ${\rm O}(1)$ gives the equation relating
 $W_1^{0,\rm J}, W_1^{1,\rm J}, W_2^{0,\rm J}$ (the latter at coincident points)
\begin{multline}
		\label{eq:Jac_1_order}
					\alpha \derivative{x} W_1^{0,\rm J}(x) - \derivative{x} W_1^{1,\rm J}(x_1) + \left(\frac{\alpha_1}{x} -\frac{\alpha_2}{1-x}\right)W_1^{1,\rm J}(x)  - {\alpha \over x (1 - x)} \\ + \alpha W_2^{0,\rm J}(x,x) + 2 \alpha W_1^{0,\rm J}(x) W_1^{1,\rm J}(x) = 0.
	\end{multline}
In relation to $W_2^{0,J}(x,x)$ herein, letting $x_1\to x_2$ and redefining  $x_2 = x$ in \eqref{eq:Jac_2_N_order} shows
		\begin{multline}\label{eq:Jac_2_N_order_onevariable}
			-\frac{1}{2}\derivative{x}W_2^{0,\rm J}(x,x) + \left(\frac{\alpha_1}{x} -\frac{\alpha_2}{1-x}\right)W_2^{0,\rm J}(x,x) -\frac{W_1^{0,\rm J}(x)}{x(1-x)} - {1 \over 1-x}\derivative{x}W_1^{0,\rm J}(x)\\  + {1 \over 2} {d^2 \over dx^2} W_1^{0,\rm J}(x) + {1 \over x} \derivative{x}W_1^{0,\rm J}(x) + 2\alpha W_2^{0,\rm J}(x,x)W_1^{0,\rm J}(x) = 0\, .
\end{multline} 
As with (\ref{w1}) and (\ref{w1L}),
introducing the expansion about $x = \infty$
\begin{equation}\label{w1J}
  W_2^{0,\rm J}(x,x) = {1 \over x^2} \sum_{p=1}^\infty {\tilde{\mu}_{p,0}^{\rm J} \over x^{p}}, \qquad
  \tilde{\mu}_{p,0}^{\rm J}  = \sum_{p_1 + q_1 = p} \mu_{(p_1,q_1),0}^{\rm J},
 \end{equation}  
together with the analogous expansion of $W_1^{0,\rm J}$ from (\ref{JT2}), we obtain from
\eqref{eq:Jac_2_N_order_onevariable}
 the recurrence
\begin{equation}\label{w2J}
\tilde{\mu}_{p,0}^{\rm J} = {1 \over  \alpha_1 + \alpha_2 +2 \alpha + 1 + p/2} \bigg ( \sum_{s=1}^{p-1} s m_{s,0}^{\rm J} - {(p-1)p \over 2} m_{p,0}^{\rm J} 
-\alpha_2 \sum_{s=1}^{p-1} \tilde{\mu}_{s,0}^{\rm J}
-2 \alpha\sum_{s=1}^{p-1} \tilde{\mu}_{s,0}^{\rm J} m_{p-s,0}^{\rm J} \bigg ).
\end{equation} 
With the help of computer algebra, we can check in low order cases that the sequence $\{\tilde{\mu}_{p,0}^{\rm J} \}_{p=1,2,\dots}$ generated by this recurrence is consistent with its relationship to 
$\{ \mu_{(p_1,q_1),0}^{\rm J} \}$ as implied by the second equation in (\ref{w1J}).

In (\ref{eq:Jac_1_order}) we have now have
$\{m_{p,0}^{\rm J} \}$ determined by 
the recurrence (\ref{Jr1}),
and $\{ \tilde{\mu}_{p,0}^{\rm J}\}$ determined by the recurrence (\ref{w2J}). Now introducing the expansion
\begin{equation}\label{y1L}
  W_1^{1,\rm J}(x \hbox{{\sout{$x$}}}) = {1 \over x} \sum_{p=2}^\infty {m_{p,1}^{\rm J} \over x^{p}}
  \end{equation}
  we see  that $\{m_{p,1}^{\rm J}\}$ can be determined by the recurrence
\begin{equation}\label{y2L}  
m_{p,1}^{\rm J} = {1 \over p + 1 + \alpha_1 + \alpha_2 +2 \alpha}
\bigg ( \alpha (p+1) m_{p,0}^{\rm J} 
- \alpha (\tilde{\mu}_{p,0} + 1) - 2 \alpha \sum_{s=1}^{p-1} m_{s,1}^{\rm J} m_{p-s,0}^{\rm J}
- \alpha_2 \sum_{s=1}^{p-1} m_{s,1}^{\rm J} \bigg ),
\end{equation}
valid for $p=1,2,\dots$ with initial condition $m_{0,1}^{\rm J}=0$. By the aid of computer algebra, it can be checked that (\ref{y2L}) correctly reproduces the values of $m_{p,1}^{\rm J}$ for $p=1$ and $p=2$ as implied by (\ref{5.2}).

\section{Application to Dyson's disordered chain}\label{S6}
\subsection{Anti-symmetric Gaussian $\beta$-ensemble in the high temperature regime}
 
Starting with the work \cite{DE02}, it has been known how to construct random tridiagonal matrices whose eigenvalue probability density function realises the classical $\beta$ ensembles and thus have functional form given by (\ref{1.2}) for appropriate $w(x)$. A systematic discussion in the context of the high temperature regime as specified by the relation (\ref{0.8}) is given in \cite{Ma20}. Our interest for subsequent application is a particular tridiagonal anti-symmetric matrix that gives rise to a variant of (\ref{1.2}) involving the Laguerre weight, but with squared variables. This is the anti-symmetric Gaussian $\beta$-ensemble introduced in \cite{DF10}. With  $\tilde{\chi}_k$ denoting the square root of the gamma distribution $\Gamma[k/2,1]$, the latter random tridiagonal matrix is specified by  
with entries directly
  above the diagonal being distributed by  
   \begin{equation}\label{A1+}   
  (\tilde{\chi}_{(N - 1)\beta/2}, \tilde{\chi}_{\beta(N - 2)/2},\dots, \tilde{\chi}_{\beta/2}).
   \end{equation}
It was shown in \cite{DF10} that the eigenvalue PDF can be explicitly determined, with the precise functional
  form depending  on the parity of $N$. 
  Replacing $N$ by $2N+1$ so the size of the matrix is odd, there is one
  zero eigenvalue, with the remaining eigenvalues coming in pairs $\{ \pm i x_j \}_{j=1}^N$, $x_j > 0$.
  Their squares $x_j^2 =: y_j$ are
  distributed according to the PDF proportional to 
   \begin{equation}\label{A1}   
   \prod_{l=1}^N y_l^{3\beta/4 - 1} e^{- y_l} \prod_{1 \le j < k \le N} | y_k - y_j|^\beta,
   \end{equation} 
   and is thus an example of the Laguerre $\beta$-ensemble with $\alpha_1 = 3\beta/4 - 1$.

As observed in the recent work \cite{Fo21}, it follows from the theory of the Laguerre $\beta$-ensemble 
with high temperature scaling (\ref{0.8}) that the anti-symmetric Gaussian $\beta$-ensemble too permits a well defined high temperature limit specified by the scaling (\ref{0.8}) {with $\alpha >0$}. Specifically, {taking the limit of (\ref{uLd}) for  $\alpha_1 \to -1$, we get that that} the limiting density of the squared eigenvalues is given in terms of a particular Whittaker function according to
\begin{equation}\label{uLdplus1}
\rho_{(1),0}^{\text{(a-s)}^2}(y;\alpha) =
{1 \over \Gamma(\alpha + 1) \Gamma(\alpha)}
{1 \over | W_{-\alpha + 1/2,0}(-y)|^2},
\end{equation}
supported on $y > 0$. This relates to the density of the eigenvalues themselves (i.e.~without squaring) by the simple relation
\begin{equation}\label{uLd+2}
\rho_{(1),0}^{\text{a-s}}(x;\alpha) =
2 x \rho_{(1),0}^{\text{(a-s)}^2}(x^2;\alpha).
\end{equation}
In particular, combining this with (\ref{Lk1}) shows
\begin{equation}\label{uLd+3}
\int_0^\infty x^2 \rho_{(1),0}^{\text{a-s}}(x;\alpha) \, dx =
\int_0^\infty y \rho_{(1),0}^{\text{(a-s)}^2}(y ;\alpha) \, dy = \alpha.
\end{equation}

The result (\ref{uLdplus1}) in the form implied by (\ref{uLd+2}) it is illustrated through numerical simulation of the eigenvalue density of the 
anti-symmetric Gaussian $\beta$-ensemble scaled by (\ref{0.8}) in Figure \ref{fig:simulations}. To tabulate (\ref{uLdplus1}) for $\alpha \not\in \Z$, use is made of the connection formula for the Whittaker function
\begin{equation}
    W_{k,\mu}(z) = -\frac{W_{-k,\mu}(-z)\Gamma(\frac{1}{2} +\mu +k)}{\Gamma(\frac{1}{2} +\mu -k)}e^{-\left( \mu +\frac{1}{2}\right)\pi i} + \frac{\Gamma(\frac{1}{2} +\mu  +k)e^{k\pi i}}{\Gamma(1 +2\mu)}M_{-k,\mu}(-z), 
    \end{equation}
where $M_{k,\mu}(z)$ is the second solution of the Whittaker equation \cite[\S 13.14(i)]{DLMF}.

\begin{figure}
    \centering
    \includegraphics[scale=0.4]{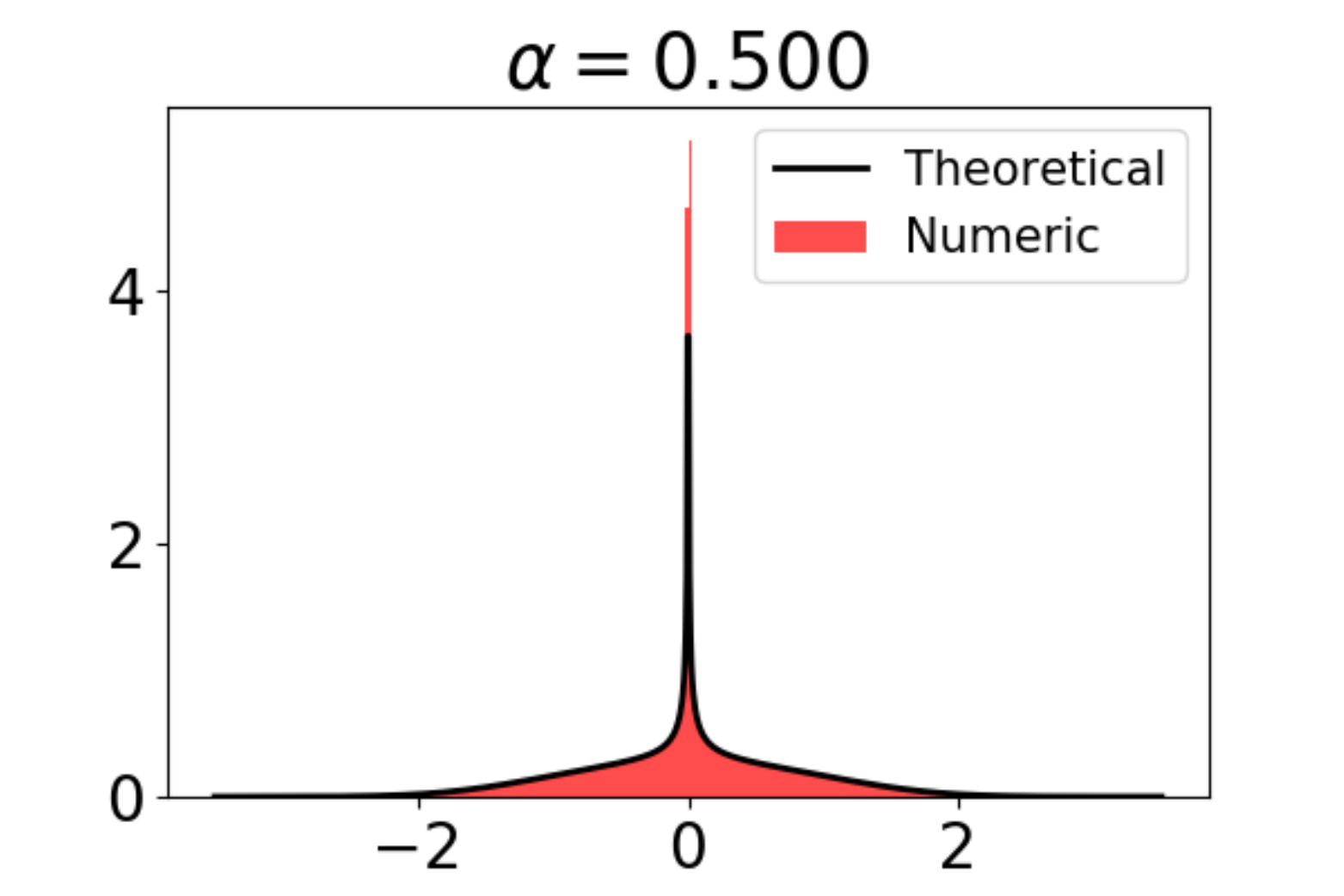}    \includegraphics[scale=0.4]{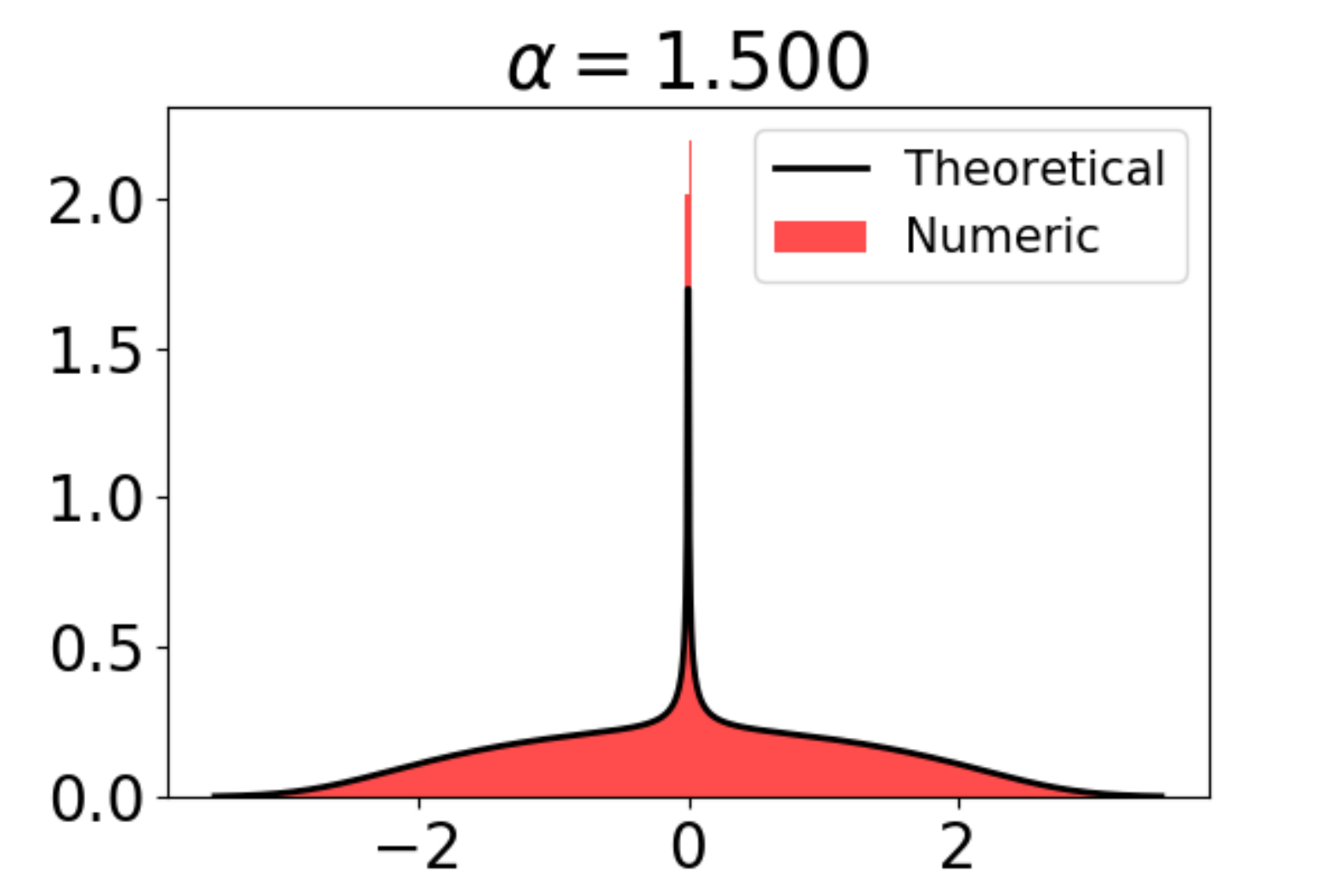}
    \includegraphics[scale=0.4]{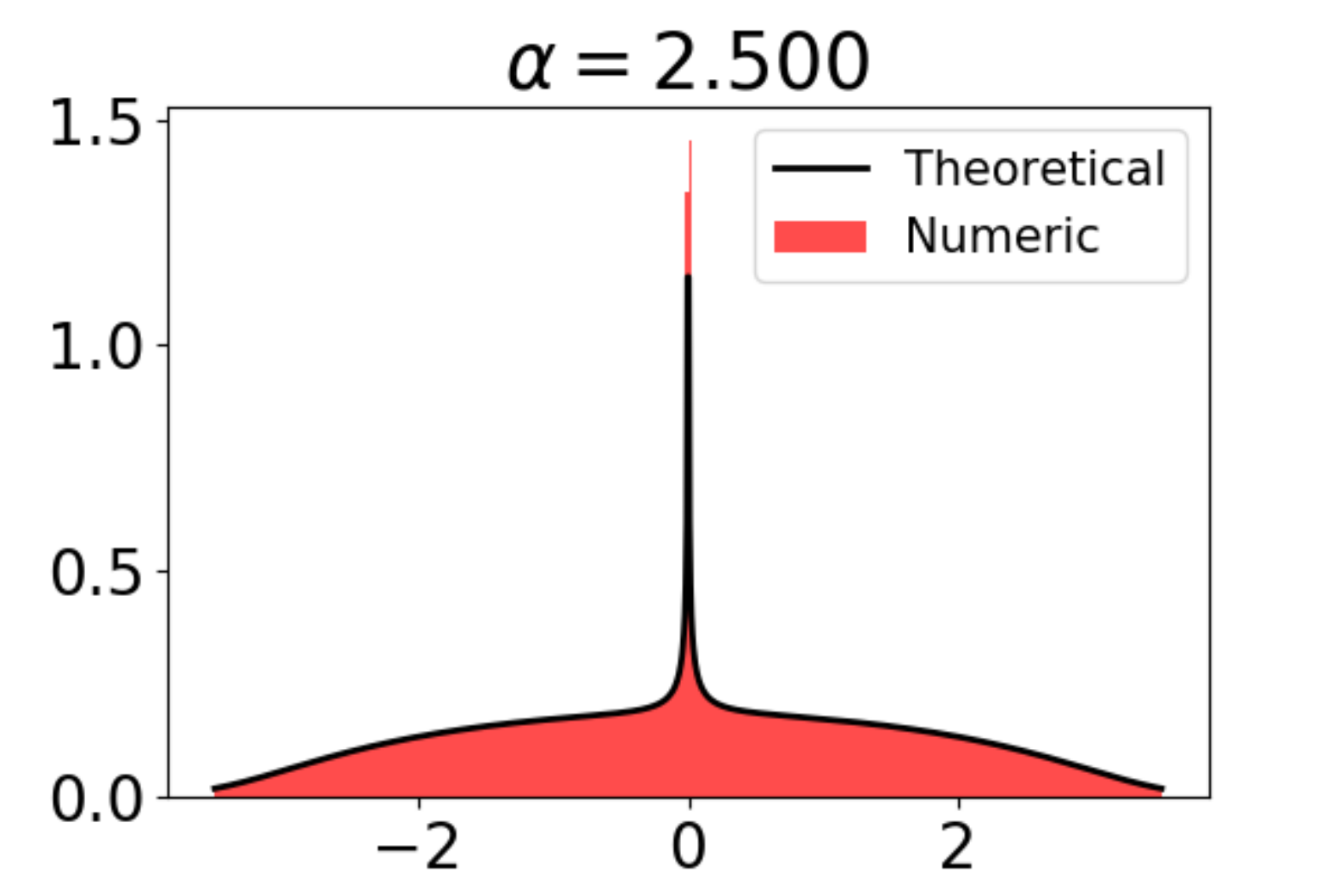}
    \includegraphics[scale=0.4]{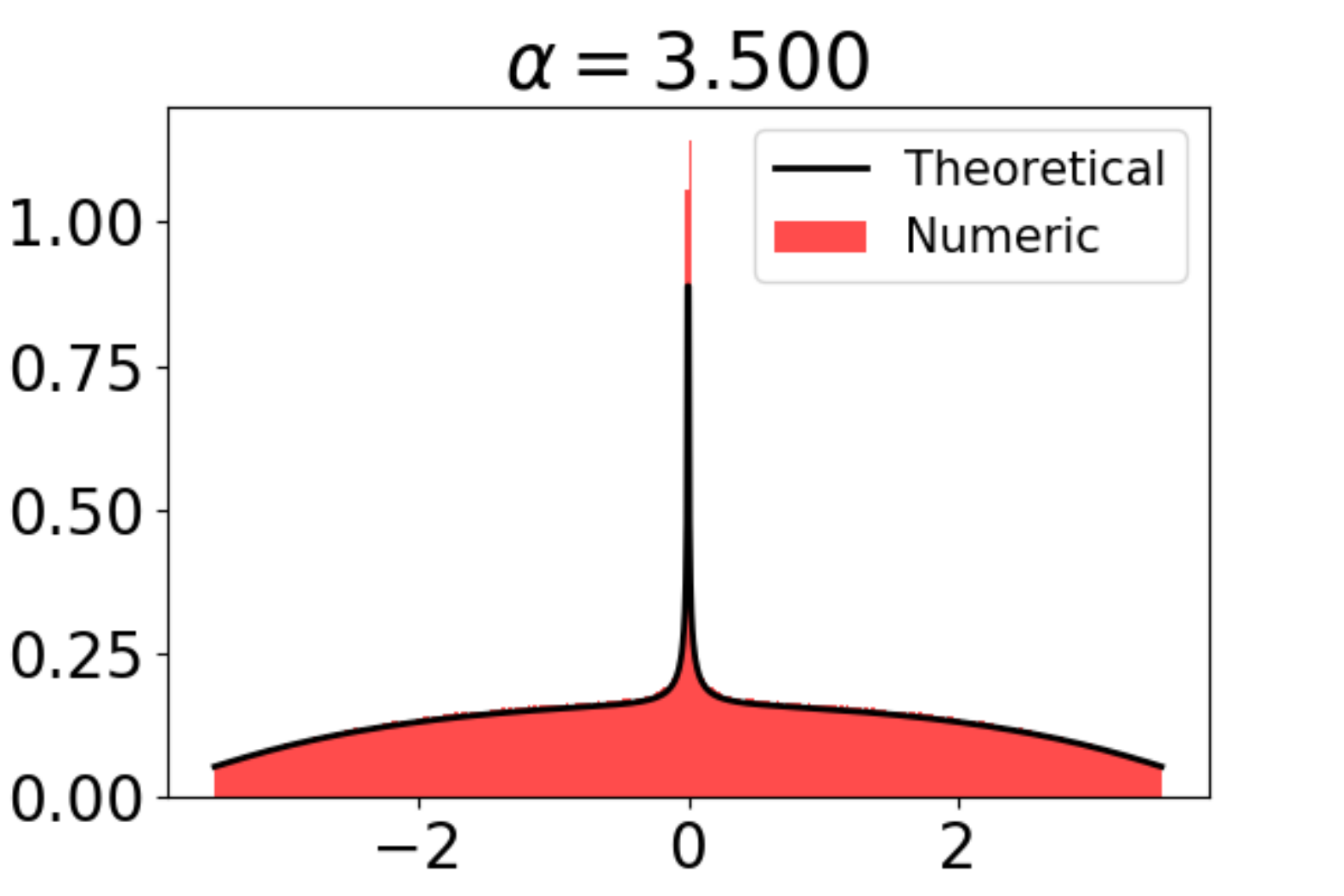}
    \caption{Simulation of anti-symmetric Gaussian $\beta$-ensemble density of states in the high temperature regime, $n=5000$, $\mbox{trials}=500\,$. {The density smoothly goes to zero outside the plotted area.}}
    \label{fig:simulations}
\end{figure}
\subsection{Anti-symmetric Gaussian $\alpha$-ensemble}
We define the anti-symmetric Gaussian $\alpha$-ensemble, {with $\alpha >0$,} in analogy with the other $\alpha$-ensembles defined in the recent work \cite{Ma20}. This is done by noting that in the high temperature regime the entries in the top left corner of the tridiagonal realisation of the classical $\beta$-ensembles are to leading order independent of the row and thus i.i.d. The prescription then is to construct a random tridiagonal matrix with these random variables. In the case of the anti-symmetric Gaussian $\beta$-ensemble, where the off diagonal distributions before the high temperature scaling (\ref{0.8}) are given by (\ref{A1+}), this gives for an element of anti-symmetric Gaussian $\alpha$-ensemble as the random $N \times N$ tridiagonal matrix
\begin{equation}
\label{eq:alphaense}
     A^\alpha_N = 	\begin{bmatrix}
		0 & \tilde{\chi}_{\alpha} \\
		-\tilde{\chi}_{\alpha} & 0 & \tilde{\chi}_{\alpha} \\
		&\ddots &\ddots &\ddots\\
		&& -\tilde{\chi}_\alpha & 0 & \tilde{\chi}_{\alpha}\\
		&&& -\tilde{\chi}_{\alpha} & 0 
	\end{bmatrix}.
\end{equation}
Here the entries below the diagonal are constrained to take on the appropriate values determined by the entries above the diagonal in accordance with the matrix being anti-symmetric; the entries above the diagonal are i.i.d.

Following the same idea as in \cite{Ma20}, the limiting mean spectral measure and mean density of states can be determined.

\begin{theorem}
\label{thm:antialpha}
Consider the matrix $A_N^{\alpha}$ in \eqref{eq:alphaense}, $\alpha \in \mathbb R^+$, then the mean spectral measure of $A_N^{\alpha}$ has density $\rho_{(1),0}^{\text{a-s}}(x;\alpha)$ as specified
in \eqref{uLdplus1} and \eqref{uLd+2},
and the mean density of states of $A_N^{\alpha}$ has density $ \mu_\alpha^{\text{a-s}}$ where
\begin{equation}
     \mu_\alpha^{\text{a-s}}(x) = {\partial \over \partial \alpha} (\alpha  \rho_{(1),0}^{\text{a-s}}(x;\alpha) ).
\end{equation}
Consequently, with $\mu_\alpha^{\text{(a-s)}^2}(y)$ the density in squared variables, $y = x^2$, we have
\begin{equation}\label{dc1}
\mu_\alpha^{\text{(a-s)}^2}(y) = 
{\partial \over \partial \alpha} \bigg (
{1 \over | \Gamma(\alpha) W_{-\alpha + 1/2,0}(-y)|^2}
\bigg ).
\end{equation}
\end{theorem}

We give a sketch of the proof, which in fact is a combination of two lemmas.

\begin{lemma}
\label{lem:spectral}
Consider the matrix $A_N^{\alpha}$ in \eqref{eq:alphaense}, $\alpha \in \mathbb R^+$, then the mean spectral measure of $A_N^{\alpha}$ has density 
$\rho_{(1),0}^{\text{a-s}}(x;\alpha)$ as specified
in \eqref{uLdplus1} and \eqref{uLd+2}.
\end{lemma}

\begin{proof}
Denote by $B_n(\beta)$ the top $n \times n$ sub-block of the random tridiagonal matrix specified by distribution of its leading diagonal (\ref{A1+}).
One just has to realise that for any fixed $\kappa\in \mathbb N$, $\kappa < n$ the $\kappa \times \kappa$ upper left block of $B_n(2\alpha/n)$ weakly converges to the corresponding one of $A_n^\alpha$. This implies that the two matrices have the same spectral measure, so applying the result of the previous subsection we get the claim.  
\end{proof}

\begin{lemma}
\label{lem:DOS}
Consider the matrix $A_N^{\alpha}$ in \eqref{eq:alphaense}, $\alpha \in \mathbb R^+$, let $v_\ell^{\text{a-s}}(\alpha)$ be the $\ell$-th moment of the mean spectral measure of  $A_N^{\alpha}$, and $w_\ell^{\text{a-s}}(\alpha)$ the $\ell$-th moment of the mean density of states. 
We have
\begin{equation}
    w_\ell^{\text{a-s}}(\alpha) = {\partial \over \partial \alpha} (\alpha v_\ell^{\text{a-s}}(\alpha)).
\end{equation}
Equivalently, with reference to the mean spectral measure, and mean density of states in squared variables
\begin{equation}\label{dc2}
    w_\ell^{\text{(a-s)}^2}(\alpha) = {\partial \over \partial \alpha} (\alpha v_\ell^{\text{(a-s)}^2}(\alpha)) =
   {\partial \over \partial \alpha} (\alpha m_{\ell,0}^{\rm L}) \,  \Big |_{\alpha_1 = -1}. 
\end{equation}
\end{lemma}

\begin{proof}
The argument of \cite[Lemma 3.1 -- Corollary 3.2]{Ma20} is valid in this case too.
\end{proof}

\begin{proof}[Proof of Theorem \ref{thm:antialpha}]
The first part of the claim follows immediately from Lemma \ref{lem:spectral}. Regarding the second part of the claim, from Lemma \ref{lem:DOS} we have that, in the same notation as before,
\begin{equation}
    w_\ell^{\text{a-s}}(\alpha) = {\partial \over \partial \alpha }(\alpha v_\ell^{\text{a-s}}(\alpha))\,.
\end{equation}
This relation must carry over to relate
 the densities  $\mu_\alpha^{\text{a-s}}(x)$ of the mean density of states and $\rho_{(1),0}^{\text{a-s}}(x;\alpha)$ of the mean spectral measure according to
\begin{equation}
 \mu_\alpha^{\text{a-s}}(x) = {\partial \over \partial \alpha } (\alpha  \rho_{(1),0}^{\text{a-s}}(x;\alpha) )\,,
\end{equation}
and the claim follows.
\end{proof}

Combining (\ref{dc2}) with (\ref{Lk1}) shows
\begin{align}\label{Lk1d}
w_1^{\text{(a-s)}^2}(\alpha) & = 2 \alpha \nonumber \\
w_2^{\text{(a-s)}^2}(\alpha) & = 2 \alpha (1 + 3 \alpha) \nonumber \\
w_3^{\text{(a-s)}^2}(\alpha) & = 2 \alpha (2 + 9 \alpha + 10 \alpha^2).
\end{align}

\subsection{Dyson's disordered chain}
As a mathematical model of a disordered system, Dyson \cite{Dy53} made a study of the distribution of the squared frequencies for $N$ coupled oscillators along a line, in the circumstance that the spring constants, and/or the masses are random variables (for some example of lattices with random initial data see \cite{KKTM,GMM} and the references therein). Let $K_j$ denote the spring constant of the $j$-th spring, and let $m_j$ denote the attached mass. With free boundary conditions it was shown in \cite{Dy53} that the allowed frequencies $\omega$ of the chain are given by the $(N-1)$ positive eigenvalues of the matrix $i \mathbf \Lambda$, where $\mathbf \Lambda$ is the $(2N - 1) \times (2N-1)$ anti-symmetric tridiagonal matrix specified by having the diagonal above the main diagonal with entries
\begin{equation}\label{dc7}
(\lambda_1^{1/2}, \lambda_2^{1/2},\dots, \lambda_{2N-1}^{1/2}), \qquad \lambda_{2j-1} = K_j/m_j, \:\lambda_{2j} = K_j/m_{j+1}.
\end{equation}
This matrix also has one zero eigenvalue, in keeping with the choice of free boundary conditions.

As observed in \cite{Dy53}, the structure  (\ref{dc7}) implies that the simplest type of disorder  is to choose $\{ \lambda_j \}$ from a common probability distribution,
giving rise to what was termed a a Type I disordered chain. Moreover, with the common probability distribution equalling the gamma distribution $\Gamma[\alpha,1/\kappa]$, Dyson was able to obtain a number of analytic results. Substituting the gamma distribution in (\ref{dc7}),
up to scaling by a factor of $1/\sqrt{\kappa}$, we see that Dyson was in fact studying matrices from the anti-symmetric Gaussian $\alpha$-ensemble  \eqref{eq:alphaense}.
One of the analytic results obtained in \cite[Eq.~(63))]{Dy53} was, in the case $\alpha \in \N$,
an explicit functional form for the integrated mean density of states in squared variables. Our
Theorem \ref{thm:antialpha} generalizes the result of Dyson by giving a special function evaluation of the mean density of states for general $\alpha > 0$.
 
 Two features of Dyson's exact solution have received particular prominence as illustrating universal features, shared by models beyond the solvable case (see the recent review \cite{Fo20} for a discussion and references).
One is the functional form of the singularity $x\to 0^+$ \cite[consequence of (72)]{Dy53}:
\begin{equation}\label{dg1}
\mu_\alpha^{(\text{a-s})^2}(x) \sim \frac{c}{x|\ln(x)|^3},
\end{equation}
for some constant $c = c_\alpha >0$, now referred to as the Dyson singularity. For the constant, Dyson's result implies that for $\alpha \in \mathbb N$
\begin{equation}\label{dg2}
c_\alpha = 2 \bigg ( {\pi^2 \over 6} - \sum_{l=1}^{\alpha - 1} {1 \over l^2} \bigg ).
\end{equation}
We can also recover this result from the explicit expression \eqref{dc1}. 
First, we require knowledge of the asymptotic behaviour of the Whittaker function \cite[Eq.~(13.14.19)]{DLMF}
for $x \to 0^+$,
\begin{equation}\label{dg3}
|W_{-\alpha + 1/2,0}(-x) | \sim {\sqrt{x} \over \Gamma(\alpha)} | \ln(x) + \psi(\alpha) +2 \gamma + i \pi|,
\end{equation}
where $\psi(\alpha)$ denotes the digamma function and $\gamma$ denotes Euler's constant. This substituted in (\ref{dc1}) show that for $x \to 0^+$
\begin{equation}\label{dg4}
\mu_\alpha^{(\text{a-s})^2}(x) \sim
{2 \over x} {\psi'(\alpha) \over | \ln(x) |^3}.
\end{equation}
From the explicit formula for the trigamma function
$$
\psi'(\alpha) = \sum_{n=0}^\infty {1 \over ( \alpha + n)^2}
$$
we see that the constant of proportionality in (\ref{dg4}) reduces to Dyson's result (\ref{dg2}) for $\alpha \in \mathbb N$.

The other prominent feature of Dyson's exact solution solution relates to the (scaled) limit $\alpha \to \infty$, which corresponds to weak disorder; see the discussion of \cite[\S 3.4]{Fo20} for more details and references. Proceeding analogously to the analysis of Remark \ref{R31} in the Gaussian case, we see that upon the scaling $x \mapsto \kappa y$ and $W_1^{0, \rm L}(x) \mapsto {1 \over \kappa} W_1^{0, \rm L}(y)$,
for $\kappa, \alpha \to \infty$ with $\kappa/\alpha = {\rm O}(1)$ \eqref{eq:Lag_N_order} reduces to the quadratic equation
$$
W_1^{0, \rm L}(y) + {1 \over y} + {\alpha \over \kappa} (W_1^{0, \rm L}(y))^2 = 0.
$$
Subject to the requirement that for large $y$ this behaves as $1/y$, the solution of this quadratic equation is
$$
W_1^{0, \rm L}(y) = {\kappa \over 2 \alpha} \Big ( 1 - (1 - 4 \alpha/\kappa y)^{1/2} \Big ).
$$
Consequently, in the same limit,
$$
\kappa \rho_{(1),0}^{\rm L}(y;\alpha_1,\alpha) \to{\kappa \over 2 \pi \alpha} (4 \alpha/ \kappa y - 1)^{1/2},
\qquad 0 < y < 4 \alpha/\kappa.
$$
But $\rho_{(1),0}^{\rm L}(y;\alpha_1,\alpha)|_{\alpha_1 = -1} = \rho_{(1),0}^{({(\rm a-s)}{}^2}(y;\alpha)$ and so according to (\ref{dc2})
\begin{equation}\label{dg5}
\kappa \mu_{\alpha}^{{(\rm a-s)}{}^2)}(\kappa y) \to {\kappa \over 2 \pi} {\partial \over \partial \alpha}
(4 \alpha/(\kappa y ) - 1)^{1/2}
= {1 \over \pi} (4 \alpha y^{1/2}/ \kappa - y)^{-1/2}, \qquad 0 < y < 4 \alpha/\kappa.
\end{equation}
With $\kappa = \alpha$, this is in precise agreement with the limiting result obtained by Dyson \cite[Eq.~(43)]{Dy53}.

\begin{acknowledgments}
	The research of PJF is part of the program of study supported
	by the Australian Research Council Centre of Excellence ACEMS,
	and the Discovery Project grant DP210102887.
	The research of GM is part of the program of study supported by the European Union's H2020 research and innovation program under the Marie Sk\l owdoska--Curie grant No. 778010 {\em  IPaDEGAN}. We thank G.~Akemann for (indirectly) facilitating this collaboration by inviting GM to speak as part of the Bielefeld-Melbourne random matrix seminar in December 2020. We thank K.D.~Trinh for alerting us that a recursive formula for the covariances $\{\mu_{(p,q),0}^{\rm G} \}$ was first given in \cite{Sp20}. 
\end{acknowledgments}

\section*{Data Availability Statement}

The code used to produce Figure \ref{fig:simulations} is available on GitHub \cite{RMT_repo}.

 \providecommand{\noopsort}[1]{}
  \providecommand{\singleletter}[1]{#1}
\bibliography{loopbib}

\begin{thebibliography}{52}%
\makeatletter
\providecommand \@ifxundefined [1]{%
 \@ifx{#1\undefined}
}%
\providecommand \@ifnum [1]{%
 \ifnum #1\expandafter \@firstoftwo
 \else \expandafter \@secondoftwo
 \fi
}%
\providecommand \@ifx [1]{%
 \ifx #1\expandafter \@firstoftwo
 \else \expandafter \@secondoftwo
 \fi
}%
\providecommand \natexlab [1]{#1}%
\providecommand \enquote  [1]{``#1''}%
\providecommand \bibnamefont  [1]{#1}%
\providecommand \bibfnamefont [1]{#1}%
\providecommand \citenamefont [1]{#1}%
\providecommand \href@noop [0]{\@secondoftwo}%
\providecommand \href [0]{\begingroup \@sanitize@url \@href}%
\providecommand \@href[1]{\@@startlink{#1}\@@href}%
\providecommand \@@href[1]{\endgroup#1\@@endlink}%
\providecommand \@sanitize@url [0]{\catcode `\\12\catcode `\$12\catcode
  `\&12\catcode `\#12\catcode `\^12\catcode `\_12\catcode `\%12\relax}%
\providecommand \@@startlink[1]{}%
\providecommand \@@endlink[0]{}%
\providecommand \url  [0]{\begingroup\@sanitize@url \@url }%
\providecommand \@url [1]{\endgroup\@href {#1}{\urlprefix }}%
\providecommand \urlprefix  [0]{URL }%
\providecommand \Eprint [0]{\href }%
\providecommand \doibase [0]{http://dx.doi.org/}%
\providecommand \selectlanguage [0]{\@gobble}%
\providecommand \bibinfo  [0]{\@secondoftwo}%
\providecommand \bibfield  [0]{\@secondoftwo}%
\providecommand \translation [1]{[#1]}%
\providecommand \BibitemOpen [0]{}%
\providecommand \bibitemStop [0]{}%
\providecommand \bibitemNoStop [0]{.\EOS\space}%
\providecommand \EOS [0]{\spacefactor3000\relax}%
\providecommand \BibitemShut  [1]{\csname bibitem#1\endcsname}%
\let\auto@bib@innerbib\@empty
\bibitem [{\citenamefont {Akemann}\ and\ \citenamefont {Byun}(2019)}]{AB19}%
  \BibitemOpen
  \bibfield  {author} {\bibinfo {author} {\bibnamefont {Akemann}, \bibfnamefont
  {G.}}\ and\ \bibinfo {author} {\bibnamefont {Byun}, \bibfnamefont {S.-S.}},\
  }\bibfield  {title} {\enquote {\bibinfo {title} {The high temperature
  crossover for general 2{D} {C}oulomb gases},}\ }\href {\doibase
  10.1007/s10955-019-02276-6} {\bibfield  {journal} {\bibinfo  {journal} {J.
  Stat. Phys.}\ }\textbf {\bibinfo {volume} {175}},\ \bibinfo {pages}
  {1043--1065} (\bibinfo {year} {2019})}\BibitemShut {NoStop}%
\bibitem [{\citenamefont {Allez}, \citenamefont {Bouchaud},\ and\ \citenamefont
  {Guionnet}(2012)}]{ABG12}%
  \BibitemOpen
  \bibfield  {author} {\bibinfo {author} {\bibnamefont {Allez}, \bibfnamefont
  {R.}}, \bibinfo {author} {\bibnamefont {Bouchaud}, \bibfnamefont {J.~P.}}, \
  and\ \bibinfo {author} {\bibnamefont {Guionnet}, \bibfnamefont {A.}},\
  }\bibfield  {title} {\enquote {\bibinfo {title} {{Invariant beta ensembles
  and the Gauss-Wigner crossover}},}\ }\href {\doibase
  10.1103/PhysRevLett.109.094102} {\bibfield  {journal} {\bibinfo  {journal}
  {Phys. Rev. Lett.}\ }\textbf {\bibinfo {volume} {109}},\ \bibinfo {pages}
  {1--5} (\bibinfo {year} {2012})}\BibitemShut {NoStop}%
\bibitem [{\citenamefont {Allez}\ \emph {et~al.}(2013)\citenamefont {Allez},
  \citenamefont {Bouchaud}, \citenamefont {Majumdar},\ and\ \citenamefont
  {Vivo}}]{ABMV13}%
  \BibitemOpen
  \bibfield  {author} {\bibinfo {author} {\bibnamefont {Allez}, \bibfnamefont
  {R.}}, \bibinfo {author} {\bibnamefont {Bouchaud}, \bibfnamefont {J.~P.}},
  \bibinfo {author} {\bibnamefont {Majumdar}, \bibfnamefont {S.~N.}}, \ and\
  \bibinfo {author} {\bibnamefont {Vivo}, \bibfnamefont {P.}},\ }\bibfield
  {title} {\enquote {\bibinfo {title} {{Invariant $\beta$-Wishart ensembles,
  crossover densities and asymptotic corrections to the Mar{\v{c}}enko-Pastur
  law}},}\ }\href {\doibase 10.1088/1751-8113/46/1/015001} {\bibfield
  {journal} {\bibinfo  {journal} {J. Phys. A Math. Theor.}\ }\textbf {\bibinfo
  {volume} {46}},\ \bibinfo {pages} {1--26} (\bibinfo {year}
  {2013})}\BibitemShut {NoStop}%
\bibitem [{\citenamefont {Allez}\ and\ \citenamefont {Guionnet}(2013)}]{AG13}%
  \BibitemOpen
  \bibfield  {author} {\bibinfo {author} {\bibnamefont {Allez}, \bibfnamefont
  {R.}}\ and\ \bibinfo {author} {\bibnamefont {Guionnet}, \bibfnamefont {A.}},\
  }\bibfield  {title} {\enquote {\bibinfo {title} {{A diffusive matrix model
  for invariant $\beta$-ensembles}},}\ }\href {\doibase 10.1214/EJP.v18-2073}
  {\bibfield  {journal} {\bibinfo  {journal} {Electron. J. Probab.}\ }\textbf
  {\bibinfo {volume} {18}},\ \bibinfo {pages} {1 -- 30} (\bibinfo {year}
  {2013})}\BibitemShut {NoStop}%
\bibitem [{\citenamefont {Askey}\ and\ \citenamefont {Wimp}(1984)}]{AW84}%
  \BibitemOpen
  \bibfield  {author} {\bibinfo {author} {\bibnamefont {Askey}, \bibfnamefont
  {R.}}\ and\ \bibinfo {author} {\bibnamefont {Wimp}, \bibfnamefont {J.}},\
  }\bibfield  {title} {\enquote {\bibinfo {title} {Associated {L}aguerre and
  {H}ermite polynomials},}\ }\href {\doibase 10.1017/S0308210500020412}
  {\bibfield  {journal} {\bibinfo  {journal} {Proceedings of the Royal Society
  of Edinburgh: Section A Mathematics}\ }\textbf {\bibinfo {volume} {96}},\
  \bibinfo {pages} {15–37} (\bibinfo {year} {1984})}\BibitemShut {NoStop}%
\bibitem [{\citenamefont {Assiotis}\ \emph {et~al.}(2020)\citenamefont
  {Assiotis}, \citenamefont {Bedert}, \citenamefont {Gunes},\ and\
  \citenamefont {Soor}}]{ABGS20}%
  \BibitemOpen
  \bibfield  {author} {\bibinfo {author} {\bibnamefont {Assiotis},
  \bibfnamefont {T.}}, \bibinfo {author} {\bibnamefont {Bedert}, \bibfnamefont
  {B.}}, \bibinfo {author} {\bibnamefont {Gunes}, \bibfnamefont {M.~A.}}, \
  and\ \bibinfo {author} {\bibnamefont {Soor}, \bibfnamefont {A.}},\
  }\href@noop {} {\enquote {\bibinfo {title} {Moments of generalized {C}auchy
  random matrices and continuous-{H}ahn polynomials},}\ } (\bibinfo {year}
  {2020}),\ \Eprint {http://arxiv.org/abs/2009.04752} {arXiv:2009.04752
  [math.PR]} \BibitemShut {NoStop}%
\bibitem [{\citenamefont {Borot}\ \emph {et~al.}(2011)\citenamefont {Borot},
  \citenamefont {Eynard}, \citenamefont {Majumdar},\ and\ \citenamefont
  {Nadal}}]{BEMN10}%
  \BibitemOpen
  \bibfield  {author} {\bibinfo {author} {\bibnamefont {Borot}, \bibfnamefont
  {G.}}, \bibinfo {author} {\bibnamefont {Eynard}, \bibfnamefont {B.}},
  \bibinfo {author} {\bibnamefont {Majumdar}, \bibfnamefont {S.~N.}}, \ and\
  \bibinfo {author} {\bibnamefont {Nadal}, \bibfnamefont {C.}},\ }\bibfield
  {title} {\enquote {\bibinfo {title} {Large deviations of the maximal
  eigenvalue of random matrices},}\ }\href {\doibase
  10.1088/1742-5468/2011/11/p11024} {\bibfield  {journal} {\bibinfo  {journal}
  {J. Stat. Mech. Theory Exp.}\ }\textbf {\bibinfo {volume} {2011}},\ \bibinfo
  {pages} {P11024} (\bibinfo {year} {2011})}\BibitemShut {NoStop}%
\bibitem [{\citenamefont {Borot}\ and\ \citenamefont {Guionnet}(2013)}]{BG12}%
  \BibitemOpen
  \bibfield  {author} {\bibinfo {author} {\bibnamefont {Borot}, \bibfnamefont
  {G.}}\ and\ \bibinfo {author} {\bibnamefont {Guionnet}, \bibfnamefont {A.}},\
  }\bibfield  {title} {\enquote {\bibinfo {title} {Asymptotic expansion of
  {$\beta$} matrix models in the one-cut regime},}\ }\href {\doibase
  10.1007/s00220-012-1619-4} {\bibfield  {journal} {\bibinfo  {journal} {Comm.
  Math. Phys.}\ }\textbf {\bibinfo {volume} {317}},\ \bibinfo {pages}
  {447--483} (\bibinfo {year} {2013})}\BibitemShut {NoStop}%
\bibitem [{\citenamefont {Brini}, \citenamefont {Mari\~{n}o},\ and\
  \citenamefont {Stevan}(2011)}]{BMS11}%
  \BibitemOpen
  \bibfield  {author} {\bibinfo {author} {\bibnamefont {Brini}, \bibfnamefont
  {A.}}, \bibinfo {author} {\bibnamefont {Mari\~{n}o}, \bibfnamefont {M.}}, \
  and\ \bibinfo {author} {\bibnamefont {Stevan}, \bibfnamefont {S.}},\
  }\bibfield  {title} {\enquote {\bibinfo {title} {The uses of the refined
  matrix model recursion},}\ }\href {\doibase 10.1063/1.3587063} {\bibfield
  {journal} {\bibinfo  {journal} {J. Math. Phys.}\ }\textbf {\bibinfo {volume}
  {52}},\ \bibinfo {pages} {052305, 24} (\bibinfo {year} {2011})}\BibitemShut
  {NoStop}%
\bibitem [{\citenamefont {Cunden}\ \emph {et~al.}(2019)\citenamefont {Cunden},
  \citenamefont {Mezzadri}, \citenamefont {O'Connell},\ and\ \citenamefont
  {Simm}}]{CMOS18}%
  \BibitemOpen
  \bibfield  {author} {\bibinfo {author} {\bibnamefont {Cunden}, \bibfnamefont
  {F.~D.}}, \bibinfo {author} {\bibnamefont {Mezzadri}, \bibfnamefont {F.}},
  \bibinfo {author} {\bibnamefont {O'Connell}, \bibfnamefont {N.}}, \ and\
  \bibinfo {author} {\bibnamefont {Simm}, \bibfnamefont {N.}},\ }\bibfield
  {title} {\enquote {\bibinfo {title} {Moments of random matrices and
  hypergeometric orthogonal polynomials},}\ }\href {\doibase
  10.1007/s00220-019-03323-9} {\bibfield  {journal} {\bibinfo  {journal} {Comm.
  Math. Phys.}\ }\textbf {\bibinfo {volume} {369}},\ \bibinfo {pages}
  {1091--1145} (\bibinfo {year} {2019})}\BibitemShut {NoStop}%
\bibitem [{{\relax DLMF}()}]{DLMF}%
  \BibitemOpen
  {\relax DLMF},\ \href {http://dlmf.nist.gov/} {\enquote {\bibinfo {title}
  {{\it NIST Digital Library of Mathematical Functions}},}\ }\bibinfo
  {howpublished} {http://dlmf.nist.gov/, Release 1.1.0 of 2020-12-15},\
  \bibinfo {note} {f.~W.~J. Olver, A.~B. {Olde Daalhuis}, D.~W. Lozier, B.~I.
  Schneider, R.~F. Boisvert, C.~W. Clark, B.~R. Miller, B.~V. Saunders, H.~S.
  Cohl, and M.~A. McClain, eds.}\BibitemShut {Stop}%
\bibitem [{\citenamefont {Drake}(2009)}]{Dr09}%
  \BibitemOpen
  \bibfield  {author} {\bibinfo {author} {\bibnamefont {Drake}, \bibfnamefont
  {D.}},\ }\bibfield  {title} {\enquote {\bibinfo {title} {The combinatorics of
  associated {H}ermite polynomials},}\ }\href {\doibase
  10.1016/j.ejc.2008.05.009} {\bibfield  {journal} {\bibinfo  {journal}
  {European J. Combin.}\ }\textbf {\bibinfo {volume} {30}},\ \bibinfo {pages}
  {1005--1021} (\bibinfo {year} {2009})}\BibitemShut {NoStop}%
\bibitem [{\citenamefont {Dumitriu}\ and\ \citenamefont
  {Edelman}(2002)}]{DE02}%
  \BibitemOpen
  \bibfield  {author} {\bibinfo {author} {\bibnamefont {Dumitriu},
  \bibfnamefont {I.}}\ and\ \bibinfo {author} {\bibnamefont {Edelman},
  \bibfnamefont {A.}},\ }\bibfield  {title} {\enquote {\bibinfo {title} {Matrix
  models for beta ensembles},}\ }\href {\doibase 10.1063/1.1507823} {\bibfield
  {journal} {\bibinfo  {journal} {J. Math. Phys.}\ }\textbf {\bibinfo {volume}
  {43}},\ \bibinfo {pages} {5830--5847} (\bibinfo {year} {2002})}\BibitemShut
  {NoStop}%
\bibitem [{\citenamefont {Dumitriu}\ and\ \citenamefont
  {Edelman}(2006)}]{DE06}%
  \BibitemOpen
  \bibfield  {author} {\bibinfo {author} {\bibnamefont {Dumitriu},
  \bibfnamefont {I.}}\ and\ \bibinfo {author} {\bibnamefont {Edelman},
  \bibfnamefont {A.}},\ }\bibfield  {title} {\enquote {\bibinfo {title} {Global
  spectrum fluctuations for the $\beta$-{H}ermite and $\beta$-laguerre
  ensembles via matrix models},}\ }\href {\doibase 10.1063/1.2200144}
  {\bibfield  {journal} {\bibinfo  {journal} {J. Math. Phys.}\ }\textbf
  {\bibinfo {volume} {47}},\ \bibinfo {pages} {063302} (\bibinfo {year}
  {2006})}\BibitemShut {NoStop}%
\bibitem [{\citenamefont {Dumitriu}\ and\ \citenamefont
  {Forrester}(2010)}]{DF10}%
  \BibitemOpen
  \bibfield  {author} {\bibinfo {author} {\bibnamefont {Dumitriu},
  \bibfnamefont {I.}}\ and\ \bibinfo {author} {\bibnamefont {Forrester},
  \bibfnamefont {P.~J.}},\ }\bibfield  {title} {\enquote {\bibinfo {title}
  {{Tridiagonal realization of the antisymmetric Gaussian $\beta$-ensemble}},}\
  }\href {\doibase 10.1063/1.3486071} {\bibfield  {journal} {\bibinfo
  {journal} {J. Math. Phys.}\ }\textbf {\bibinfo {volume} {51}} (\bibinfo
  {year} {2010}),\ 10.1063/1.3486071}\BibitemShut {NoStop}%
\bibitem [{\citenamefont {Duy}\ and\ \citenamefont {Shirai}(2015)}]{DS15}%
  \BibitemOpen
  \bibfield  {author} {\bibinfo {author} {\bibnamefont {Duy}, \bibfnamefont
  {T.~K.}}\ and\ \bibinfo {author} {\bibnamefont {Shirai}, \bibfnamefont
  {T.}},\ }\bibfield  {title} {\enquote {\bibinfo {title} {The mean spectral
  measures of random {J}acobi matrices related to {G}aussian beta ensembles},}\
  }\href {\doibase 10.1214/ECP.v20-4252} {\bibfield  {journal} {\bibinfo
  {journal} {Electron. Commun. Probab.}\ }\textbf {\bibinfo {volume} {20}},\
  \bibinfo {pages} {no. 68, 13} (\bibinfo {year} {2015})}\BibitemShut {NoStop}%
\bibitem [{\citenamefont {Dyson}(1953)}]{Dy53}%
  \BibitemOpen
  \bibfield  {author} {\bibinfo {author} {\bibnamefont {Dyson}, \bibfnamefont
  {F.~J.}},\ }\bibfield  {title} {\enquote {\bibinfo {title} {The dynamics of a
  disordered linear chain},}\ }\href@noop {} {\bibfield  {journal} {\bibinfo
  {journal} {Phys. Rev. (2)}\ }\textbf {\bibinfo {volume} {92}},\ \bibinfo
  {pages} {1331--1338} (\bibinfo {year} {1953})}\BibitemShut {NoStop}%
\bibitem [{\citenamefont {Eynard}\ and\ \citenamefont {Marchal}(2009)}]{EM09}%
  \BibitemOpen
  \bibfield  {author} {\bibinfo {author} {\bibnamefont {Eynard}, \bibfnamefont
  {B.}}\ and\ \bibinfo {author} {\bibnamefont {Marchal}, \bibfnamefont {O.}},\
  }\bibfield  {title} {\enquote {\bibinfo {title} {Topological expansion of the
  {B}ethe ansatz, and non-commutative algebraic geometry},}\ }\href {\doibase
  10.1088/1126-6708/2009/03/094} {\bibfield  {journal} {\bibinfo  {journal} {J.
  High Energy Phys.}\ ,\ \bibinfo {pages} {094, 52}} (\bibinfo {year}
  {2009})}\BibitemShut {NoStop}%
\bibitem [{\citenamefont {Forrester}(2010)}]{Fo10}%
  \BibitemOpen
  \bibfield  {author} {\bibinfo {author} {\bibnamefont {Forrester},
  \bibfnamefont {P.~J.}},\ }\href {\doibase 10.1515/9781400835416} {\emph
  {\bibinfo {title} {Log-gases and random matrices}}},\ \bibinfo {series}
  {London Mathematical Society Monographs Series}, Vol.~\bibinfo {volume} {34}\
  (\bibinfo  {publisher} {Princeton University Press, Princeton, NJ},\ \bibinfo
  {year} {2010})\ pp.\ \bibinfo {pages} {xiv+791}\BibitemShut {NoStop}%
\bibitem [{\citenamefont {Forrester}(2021{\natexlab{a}})}]{Fo20}%
  \BibitemOpen
  \bibfield  {author} {\bibinfo {author} {\bibnamefont {Forrester},
  \bibfnamefont {P.~J.}},\ }\bibfield  {title} {\enquote {\bibinfo {title}
  {Differential {I}dentities for the {S}tructure {F}unction of {S}ome {R}andom
  {M}atrix {E}nsembles},}\ }\href {\doibase 10.1007/s10955-021-02767-5}
  {\bibfield  {journal} {\bibinfo  {journal} {J. Stat. Phys.}\ }\textbf
  {\bibinfo {volume} {183}},\ \bibinfo {pages} {Paper No. 33} (\bibinfo {year}
  {2021}{\natexlab{a}})}\BibitemShut {NoStop}%
\bibitem [{\citenamefont {Forrester}(2021{\natexlab{b}})}]{Fo21}%
  \BibitemOpen
  \bibfield  {author} {\bibinfo {author} {\bibnamefont {Forrester},
  \bibfnamefont {P.~J.}},\ }\href@noop {} {\enquote {\bibinfo {title} {Dyson's
  disordered linear chain from a random matrix theory viewpoint},}\ } (\bibinfo
  {year} {2021}{\natexlab{b}}),\ \Eprint {http://arxiv.org/abs/2101.02339}
  {arXiv:2101.02339 [math-ph]} \BibitemShut {NoStop}%
\bibitem [{\citenamefont {Forrester}\ and\ \citenamefont {Kumar}(2019)}]{FK19}%
  \BibitemOpen
  \bibfield  {author} {\bibinfo {author} {\bibnamefont {Forrester},
  \bibfnamefont {P.~J.}}\ and\ \bibinfo {author} {\bibnamefont {Kumar},
  \bibfnamefont {S.}},\ }\bibfield  {title} {\enquote {\bibinfo {title}
  {Recursion scheme for the largest $\beta$-{W}ishart-{L}aguerre eigenvalue and
  {L}andauer conductance in quantum transport},}\ }\href {\doibase
  10.1088/1751-8121/ab433c} {\bibfield  {journal} {\bibinfo  {journal} {J.
  Phys. A}\ }\textbf {\bibinfo {volume} {52}},\ \bibinfo {pages} {42LT02}
  (\bibinfo {year} {2019})}\BibitemShut {NoStop}%
\bibitem [{\citenamefont {Forrester}\ and\ \citenamefont
  {Kumar}(2020{\natexlab{a}})}]{FK20a}%
  \BibitemOpen
  \bibfield  {author} {\bibinfo {author} {\bibnamefont {Forrester},
  \bibfnamefont {P.~J.}}\ and\ \bibinfo {author} {\bibnamefont {Kumar},
  \bibfnamefont {S.}},\ }\href@noop {} {\enquote {\bibinfo {title} {Computable
  structural formulas for the distribution of the $\beta$-{J}acobi edge
  eigenvalues},}\ } (\bibinfo {year} {2020}{\natexlab{a}}),\ \Eprint
  {http://arxiv.org/abs/2006.02238} {arXiv:2006.02238 [math-ph]} \BibitemShut
  {NoStop}%
\bibitem [{\citenamefont {Forrester}\ and\ \citenamefont
  {Kumar}(2020{\natexlab{b}})}]{FK20b}%
  \BibitemOpen
  \bibfield  {author} {\bibinfo {author} {\bibnamefont {Forrester},
  \bibfnamefont {P.~J.}}\ and\ \bibinfo {author} {\bibnamefont {Kumar},
  \bibfnamefont {S.}},\ }\href@noop {} {\enquote {\bibinfo {title}
  {Differential recurrences for the distribution of the trace of the
  $\beta$-{J}acobi ensemble},}\ } (\bibinfo {year} {2020}{\natexlab{b}}),\
  \Eprint {http://arxiv.org/abs/2011.00787} {arXiv:2011.00787 [math-ph]}
  \BibitemShut {NoStop}%
\bibitem [{\citenamefont {Forrester}\ and\ \citenamefont
  {Rahman}(2020)}]{FR20}%
  \BibitemOpen
  \bibfield  {author} {\bibinfo {author} {\bibnamefont {Forrester},
  \bibfnamefont {P.~J.}}\ and\ \bibinfo {author} {\bibnamefont {Rahman},
  \bibfnamefont {A.~A.}},\ }\href@noop {} {\enquote {\bibinfo {title}
  {Relations between moments for the {J}acobi and {C}auchy random matrix
  ensembles},}\ } (\bibinfo {year} {2020}),\ \Eprint
  {http://arxiv.org/abs/2011.07856} {arXiv:2011.07856 [math-ph]} \BibitemShut
  {NoStop}%
\bibitem [{\citenamefont {Forrester}, \citenamefont {Rahman},\ and\
  \citenamefont {Witte}(2017)}]{FRW17}%
  \BibitemOpen
  \bibfield  {author} {\bibinfo {author} {\bibnamefont {Forrester},
  \bibfnamefont {P.~J.}}, \bibinfo {author} {\bibnamefont {Rahman},
  \bibfnamefont {A.~A.}}, \ and\ \bibinfo {author} {\bibnamefont {Witte},
  \bibfnamefont {N.~S.}},\ }\bibfield  {title} {\enquote {\bibinfo {title}
  {Large {$N$} expansions for the {L}aguerre and {J}acobi {$\beta$}-ensembles
  from the loop equations},}\ }\href {\doibase 10.1063/1.4997778} {\bibfield
  {journal} {\bibinfo  {journal} {J. Math. Phys.}\ }\textbf {\bibinfo {volume}
  {58}},\ \bibinfo {pages} {113303, 25} (\bibinfo {year} {2017})}\BibitemShut
  {NoStop}%
\bibitem [{\citenamefont {Gisonni}, \citenamefont {Grava},\ and\ \citenamefont
  {Ruzza}(2020)}]{GGR20}%
  \BibitemOpen
  \bibfield  {author} {\bibinfo {author} {\bibnamefont {Gisonni}, \bibfnamefont
  {M.}}, \bibinfo {author} {\bibnamefont {Grava}, \bibfnamefont {T.}}, \ and\
  \bibinfo {author} {\bibnamefont {Ruzza}, \bibfnamefont {G.}},\ }\bibfield
  {title} {\enquote {\bibinfo {title} {Laguerre ensemble: correlators,
  {H}urwitz numbers and {H}odge integrals},}\ }\href {\doibase
  10.1007/s00023-020-00922-4} {\bibfield  {journal} {\bibinfo  {journal} {Ann.
  Henri Poincar\'{e}}\ }\textbf {\bibinfo {volume} {21}},\ \bibinfo {pages}
  {3285--3339} (\bibinfo {year} {2020})}\BibitemShut {NoStop}%
\bibitem [{\citenamefont {Gisonni}, \citenamefont {Grava},\ and\ \citenamefont
  {Ruzza}(2021)}]{GGR20b}%
  \BibitemOpen
  \bibfield  {author} {\bibinfo {author} {\bibnamefont {Gisonni}, \bibfnamefont
  {M.}}, \bibinfo {author} {\bibnamefont {Grava}, \bibfnamefont {T.}}, \ and\
  \bibinfo {author} {\bibnamefont {Ruzza}, \bibfnamefont {G.}},\ }\bibfield
  {title} {\enquote {\bibinfo {title} {Jacobi {E}nsemble, {H}urwitz {N}umbers
  and {W}ilson {P}olynomials},}\ }\href {\doibase 10.1007/s11005-021-01396-z}
  {\bibfield  {journal} {\bibinfo  {journal} {Lett. Math. Phys.}\ }\textbf
  {\bibinfo {volume} {111}},\ \bibinfo {pages} {Paper No. 67} (\bibinfo {year}
  {2021})}\BibitemShut {NoStop}%
\bibitem [{\citenamefont {Gorin}\ and\ \citenamefont {Kleptsyn}(2020)}]{GK20}%
  \BibitemOpen
  \bibfield  {author} {\bibinfo {author} {\bibnamefont {Gorin}, \bibfnamefont
  {V.}}\ and\ \bibinfo {author} {\bibnamefont {Kleptsyn}, \bibfnamefont {V.}},\
  }\href@noop {} {\enquote {\bibinfo {title} {Universal objects of the infinite
  beta random matrix theory},}\ } (\bibinfo {year} {2020}),\ \Eprint
  {http://arxiv.org/abs/2009.02006} {arXiv:2009.02006 [math.PR]} \BibitemShut
  {NoStop}%
\bibitem [{\citenamefont {Grava}\ \emph {et~al.}(2021)\citenamefont {Grava},
  \citenamefont {Kriecherbauer}, \citenamefont {Mazzuca},\ and\ \citenamefont
  {McLaughlin}}]{KKTM}%
  \BibitemOpen
  \bibfield  {author} {\bibinfo {author} {\bibnamefont {Grava}, \bibfnamefont
  {T.}}, \bibinfo {author} {\bibnamefont {Kriecherbauer}, \bibfnamefont {T.}},
  \bibinfo {author} {\bibnamefont {Mazzuca}, \bibfnamefont {G.}}, \ and\
  \bibinfo {author} {\bibnamefont {McLaughlin}, \bibfnamefont {K.~D. T.-R.}},\
  }\bibfield  {title} {\enquote {\bibinfo {title} {Correlation {F}unctions for
  a {C}hain of {S}hort {R}ange {O}scillators},}\ }\href {\doibase
  10.1007/s10955-021-02735-z} {\bibfield  {journal} {\bibinfo  {journal} {J.
  Stat. Phys.}\ }\textbf {\bibinfo {volume} {183}},\ \bibinfo {pages} {1}
  (\bibinfo {year} {2021})}\BibitemShut {NoStop}%
\bibitem [{\citenamefont {Grava}\ \emph {et~al.}(2020)\citenamefont {Grava},
  \citenamefont {Maspero}, \citenamefont {Mazzuca},\ and\ \citenamefont
  {Ponno}}]{GMM}%
  \BibitemOpen
  \bibfield  {author} {\bibinfo {author} {\bibnamefont {Grava}, \bibfnamefont
  {T.}}, \bibinfo {author} {\bibnamefont {Maspero}, \bibfnamefont {A.}},
  \bibinfo {author} {\bibnamefont {Mazzuca}, \bibfnamefont {G.}}, \ and\
  \bibinfo {author} {\bibnamefont {Ponno}, \bibfnamefont {A.}},\ }\bibfield
  {title} {\enquote {\bibinfo {title} {Adiabatic invariants for the {FPUT} and
  {T}oda chain in the thermodynamic limit},}\ }\href {\doibase
  10.1007/s00220-020-03866-2} {\bibfield  {journal} {\bibinfo  {journal} {Comm.
  Math. Phys.}\ }\textbf {\bibinfo {volume} {380}},\ \bibinfo {pages}
  {811--851} (\bibinfo {year} {2020})}\BibitemShut {NoStop}%
\bibitem [{\citenamefont {Hardy}\ and\ \citenamefont {Lambert}(2021)}]{HL20}%
  \BibitemOpen
  \bibfield  {author} {\bibinfo {author} {\bibnamefont {Hardy}, \bibfnamefont
  {A.}}\ and\ \bibinfo {author} {\bibnamefont {Lambert}, \bibfnamefont {G.}},\
  }\bibfield  {title} {\enquote {\bibinfo {title} {C{LT} for {C}ircular
  beta-{E}nsembles at high temperature},}\ }\href {\doibase
  10.1016/j.jfa.2020.108869} {\bibfield  {journal} {\bibinfo  {journal} {J.
  Funct. Anal.}\ }\textbf {\bibinfo {volume} {280}},\ \bibinfo {pages} {108869}
  (\bibinfo {year} {2021})}\BibitemShut {NoStop}%
\bibitem [{\citenamefont {Ismail}\ and\ \citenamefont {Masson}(1991)}]{IM91}%
  \BibitemOpen
  \bibfield  {author} {\bibinfo {author} {\bibnamefont {Ismail}, \bibfnamefont
  {M.~E.~H.}}\ and\ \bibinfo {author} {\bibnamefont {Masson}, \bibfnamefont
  {D.~R.}},\ }\bibfield  {title} {\enquote {\bibinfo {title} {Two families of
  orthogonal polynomials related to {J}acobi polynomials},}\ }in\ \href
  {\doibase 10.1216/rmjm/1181073013} {\emph {\bibinfo {booktitle} {Proceedings
  of the {U}.{S}.-{W}estern {E}urope {R}egional {C}onference on {P}ad\'{e}
  {A}pproximants and {R}elated {T}opics ({B}oulder, {CO}, 1988)}}},\
  Vol.~\bibinfo {volume} {21}\ (\bibinfo {year} {1991})\ pp.\ \bibinfo {pages}
  {359--375}\BibitemShut {NoStop}%
\bibitem [{\citenamefont {Johansson}(1998)}]{Jo98}%
  \BibitemOpen
  \bibfield  {author} {\bibinfo {author} {\bibnamefont {Johansson},
  \bibfnamefont {K.}},\ }\bibfield  {title} {\enquote {\bibinfo {title} {On
  fluctuations of eigenvalues of random {H}ermitian matrices},}\ }\href
  {\doibase 10.1215/S0012-7094-98-09108-6} {\bibfield  {journal} {\bibinfo
  {journal} {Duke Math. J.}\ }\textbf {\bibinfo {volume} {91}},\ \bibinfo
  {pages} {151--204} (\bibinfo {year} {1998})}\BibitemShut {NoStop}%
\bibitem [{\citenamefont {Kumar}(2019)}]{Ku19}%
  \BibitemOpen
  \bibfield  {author} {\bibinfo {author} {\bibnamefont {Kumar}, \bibfnamefont
  {S.}},\ }\bibfield  {title} {\enquote {\bibinfo {title} {{Recursion for the
  Smallest Eigenvalue Density of $\beta$ -Wishart–Laguerre Ensemble}},}\
  }\href {\doibase 10.1007/s10955-019-02245-z} {\bibfield  {journal} {\bibinfo
  {journal} {J. Stat. Phys.}\ }\textbf {\bibinfo {volume} {175}},\ \bibinfo
  {pages} {126--149} (\bibinfo {year} {2019})}\BibitemShut {NoStop}%
\bibitem [{\citenamefont {Letessier}(1993)}]{Le93}%
  \BibitemOpen
  \bibfield  {author} {\bibinfo {author} {\bibnamefont {Letessier},
  \bibfnamefont {J.}},\ }\bibfield  {title} {\enquote {\bibinfo {title} {On
  co-recursive associated {L}aguerre polynomials},}\ }in\ \href {\doibase
  10.1016/0377-0427(93)90143-Y} {\emph {\bibinfo {booktitle} {Proceedings of
  the {S}eventh {S}panish {S}ymposium on {O}rthogonal {P}olynomials and
  {A}pplications ({VII} {SPOA}) ({G}ranada, 1991)}}},\ Vol.~\bibinfo {volume}
  {49}\ (\bibinfo {year} {1993})\ pp.\ \bibinfo {pages} {127--136}\BibitemShut
  {NoStop}%
\bibitem [{\citenamefont {Mazzuca}(2021{\natexlab{a}})}]{Ma20}%
  \BibitemOpen
  \bibfield  {author} {\bibinfo {author} {\bibnamefont {Mazzuca}, \bibfnamefont
  {G.}},\ }\href@noop {} {\enquote {\bibinfo {title} {On the mean density of
  states of some matrices related to the beta ensembles and an application to
  the {T}oda lattice},}\ } (\bibinfo {year} {2021}{\natexlab{a}}),\ \Eprint
  {http://arxiv.org/abs/2008.04604} {arXiv:2008.04604 [math.SP]} \BibitemShut
  {NoStop}%
\bibitem [{\citenamefont {Mazzuca}(2021{\natexlab{b}})}]{RMT_repo}%
  \BibitemOpen
  \bibfield  {author} {\bibinfo {author} {\bibnamefont {Mazzuca}, \bibfnamefont
  {G.}},\ }\href {\doibase 10.5281/zenodo.4911172} {\enquote {\bibinfo {title}
  {Random matrix ensemble},}\ } (\bibinfo {year} {2021}{\natexlab{b}}),\
  \bibinfo {note} {available at
  \url{https://github.com/gmazzuca/Random_Matrix_Alpha/releases/tag/v1.0.0}}\BibitemShut
  {NoStop}%
\bibitem [{\citenamefont {Mezzadri}, \citenamefont {Reynolds},\ and\
  \citenamefont {Winn}(2017)}]{MRW15}%
  \BibitemOpen
  \bibfield  {author} {\bibinfo {author} {\bibnamefont {Mezzadri},
  \bibfnamefont {F.}}, \bibinfo {author} {\bibnamefont {Reynolds},
  \bibfnamefont {A.~K.}}, \ and\ \bibinfo {author} {\bibnamefont {Winn},
  \bibfnamefont {B.}},\ }\bibfield  {title} {\enquote {\bibinfo {title}
  {{Moments of the eigenvalue densities and of the secular coefficients of
  $\beta$-ensembles}},}\ }\href {\doibase 10.1088/1361-6544/aa518c} {\bibfield
  {journal} {\bibinfo  {journal} {Nonlinearity}\ }\textbf {\bibinfo {volume}
  {30}},\ \bibinfo {pages} {1034--1057} (\bibinfo {year} {2017})}\BibitemShut
  {NoStop}%
\bibitem [{\citenamefont {Mironov}\ \emph {et~al.}(2012)\citenamefont
  {Mironov}, \citenamefont {Morozov}, \citenamefont {Popolitov},\ and\
  \citenamefont {Shakirov}}]{MMPS12}%
  \BibitemOpen
  \bibfield  {author} {\bibinfo {author} {\bibnamefont {Mironov}, \bibfnamefont
  {A.~D.}}, \bibinfo {author} {\bibnamefont {Morozov}, \bibfnamefont {A.~Y.}},
  \bibinfo {author} {\bibnamefont {Popolitov}, \bibfnamefont {A.~V.}}, \ and\
  \bibinfo {author} {\bibnamefont {Shakirov}, \bibfnamefont {S.~R.}},\
  }\bibfield  {title} {\enquote {\bibinfo {title} {{Resolvents and
  Seiberg-Witten representation for a Gaussian $\beta$-ensemble}},}\ }\href
  {\doibase 10.1007/s11232-012-0049-y} {\bibfield  {journal} {\bibinfo
  {journal} {Theor. Math. Phys.}\ }\textbf {\bibinfo {volume} {171}},\ \bibinfo
  {pages} {505--522} (\bibinfo {year} {2012})}\BibitemShut {NoStop}%
\bibitem [{\citenamefont {Nakano}\ and\ \citenamefont {Trinh}(2018)}]{NT18}%
  \BibitemOpen
  \bibfield  {author} {\bibinfo {author} {\bibnamefont {Nakano}, \bibfnamefont
  {F.}}\ and\ \bibinfo {author} {\bibnamefont {Trinh}, \bibfnamefont {K.~D.}},\
  }\bibfield  {title} {\enquote {\bibinfo {title} {Gaussian beta ensembles at
  high temperature: eigenvalue fluctuations and bulk statistics},}\ }\href
  {\doibase 10.1007/s10955-018-2131-9} {\bibfield  {journal} {\bibinfo
  {journal} {J. Stat. Phys.}\ }\textbf {\bibinfo {volume} {173}},\ \bibinfo
  {pages} {295--321} (\bibinfo {year} {2018})}\BibitemShut {NoStop}%
\bibitem [{\citenamefont {Nevai}(1979)}]{Ne79}%
  \BibitemOpen
  \bibfield  {author} {\bibinfo {author} {\bibnamefont {Nevai}, \bibfnamefont
  {P.~G.}},\ }\bibfield  {title} {\enquote {\bibinfo {title} {On orthogonal
  polynomials},}\ }\href {\doibase 10.1016/0021-9045(79)90031-5} {\bibfield
  {journal} {\bibinfo  {journal} {J. Approx. Theory}\ }\textbf {\bibinfo
  {volume} {25}},\ \bibinfo {pages} {34--37} (\bibinfo {year}
  {1979})}\BibitemShut {NoStop}%
\bibitem [{\citenamefont {Pastur}\ and\ \citenamefont
  {Shcherbina}(2011)}]{PS11}%
  \BibitemOpen
  \bibfield  {author} {\bibinfo {author} {\bibnamefont {Pastur}, \bibfnamefont
  {L.}}\ and\ \bibinfo {author} {\bibnamefont {Shcherbina}, \bibfnamefont
  {M.}},\ }\href {\doibase 10.1090/surv/171} {\emph {\bibinfo {title}
  {Eigenvalue distribution of large random matrices}}},\ \bibinfo {series}
  {Mathematical Surveys and Monographs}, Vol.\ \bibinfo {volume} {171}\
  (\bibinfo  {publisher} {American Mathematical Society, Providence, RI},\
  \bibinfo {year} {2011})\ pp.\ \bibinfo {pages} {xiv+632}\BibitemShut
  {NoStop}%
\bibitem [{\citenamefont {Rahman}\ and\ \citenamefont
  {Forrester}(2020)}]{RF19}%
  \BibitemOpen
  \bibfield  {author} {\bibinfo {author} {\bibnamefont {Rahman}, \bibfnamefont
  {A.~A.}}\ and\ \bibinfo {author} {\bibnamefont {Forrester}, \bibfnamefont
  {P.~J.}},\ }\bibfield  {title} {\enquote {\bibinfo {title} {{Linear
  differential equations for the resolvents of the classical matrix
  ensembles}},}\ }\href {\doibase 10.1142/S2010326322500034} {\bibfield
  {journal} {\bibinfo  {journal} {Random Matrices Theory Appl.}\ } (\bibinfo
  {year} {2020}),\ 10.1142/S2010326322500034}\BibitemShut {NoStop}%
\bibitem [{\citenamefont {Spohn}(2020)}]{Sp20}%
  \BibitemOpen
  \bibfield  {author} {\bibinfo {author} {\bibnamefont {Spohn}, \bibfnamefont
  {H.}},\ }\bibfield  {title} {\enquote {\bibinfo {title} {Generalized {G}ibbs
  ensembles of the classical {T}oda chain},}\ }\href {\doibase
  10.1007/s10955-019-02320-5} {\bibfield  {journal} {\bibinfo  {journal} {J.
  Stat. Phys.}\ }\textbf {\bibinfo {volume} {180}},\ \bibinfo {pages} {4--22}
  (\bibinfo {year} {2020})}\BibitemShut {NoStop}%
\bibitem [{\citenamefont {Trinh}\ and\ \citenamefont {Trinh}(2019)}]{TT19}%
  \BibitemOpen
  \bibfield  {author} {\bibinfo {author} {\bibnamefont {Trinh}, \bibfnamefont
  {H.~D.}}\ and\ \bibinfo {author} {\bibnamefont {Trinh}, \bibfnamefont
  {K.~D.}},\ }\href@noop {} {\enquote {\bibinfo {title} {Beta {L}aguerre
  ensembles in global regime},}\ } (\bibinfo {year} {2019}),\ \Eprint
  {http://arxiv.org/abs/1907.12267} {arXiv:1907.12267 [math.PR]} \BibitemShut
  {NoStop}%
\bibitem [{\citenamefont {Trinh}\ and\ \citenamefont {Trinh}(2020)}]{TT20}%
  \BibitemOpen
  \bibfield  {author} {\bibinfo {author} {\bibnamefont {Trinh}, \bibfnamefont
  {H.~D.}}\ and\ \bibinfo {author} {\bibnamefont {Trinh}, \bibfnamefont
  {K.~D.}},\ }\href@noop {} {\enquote {\bibinfo {title} {Beta {J}acobi
  ensembles and associated jacobi polynomials},}\ } (\bibinfo {year} {2020}),\
  \Eprint {http://arxiv.org/abs/2005.01100} {arXiv:2005.01100 [math.PR]}
  \BibitemShut {NoStop}%
\bibitem [{\citenamefont {Trinh}(2019)}]{Tr17}%
  \BibitemOpen
  \bibfield  {author} {\bibinfo {author} {\bibnamefont {Trinh}, \bibfnamefont
  {K.~D.}},\ }\bibfield  {title} {\enquote {\bibinfo {title} {Global spectrum
  fluctuations for {G}aussian beta ensembles: a {M}artingale approach},}\
  }\href {\doibase 10.1007/s10959-017-0794-9} {\bibfield  {journal} {\bibinfo
  {journal} {J. Theoret. Probab.}\ }\textbf {\bibinfo {volume} {32}},\ \bibinfo
  {pages} {1420--1437} (\bibinfo {year} {2019})}\BibitemShut {NoStop}%
\bibitem [{\citenamefont {Wimp}(1987)}]{Wi87}%
  \BibitemOpen
  \bibfield  {author} {\bibinfo {author} {\bibnamefont {Wimp}, \bibfnamefont
  {J.}},\ }\bibfield  {title} {\enquote {\bibinfo {title} {Explicit formulas
  for the associated {J}acobi polynomials and some applications},}\ }\href
  {\doibase 10.4153/CJM-1987-050-4} {\bibfield  {journal} {\bibinfo  {journal}
  {Canad. J. Math.}\ }\textbf {\bibinfo {volume} {39}},\ \bibinfo {pages}
  {983--1000} (\bibinfo {year} {1987})}\BibitemShut {NoStop}%
\bibitem [{\citenamefont {Witte}\ and\ \citenamefont {Forrester}(2014)}]{WF14}%
  \BibitemOpen
  \bibfield  {author} {\bibinfo {author} {\bibnamefont {Witte}, \bibfnamefont
  {N.~S.}}\ and\ \bibinfo {author} {\bibnamefont {Forrester}, \bibfnamefont
  {P.~J.}},\ }\bibfield  {title} {\enquote {\bibinfo {title} {{Moments of the
  Gaussian $\beta$ ensembles and the large-N expansion of the densities}},}\
  }\href {\doibase 10.1063/1.4886477} {\bibfield  {journal} {\bibinfo
  {journal} {J. Math. Phys.}\ }\textbf {\bibinfo {volume} {55}},\ \bibinfo
  {pages} {083302} (\bibinfo {year} {2014})}\BibitemShut {NoStop}%
\bibitem [{\citenamefont {Witte}\ and\ \citenamefont {Forrester}(2015)}]{WF15}%
  \BibitemOpen
  \bibfield  {author} {\bibinfo {author} {\bibnamefont {Witte}, \bibfnamefont
  {N.~S.}}\ and\ \bibinfo {author} {\bibnamefont {Forrester}, \bibfnamefont
  {P.~J.}},\ }\bibfield  {title} {\enquote {\bibinfo {title} {Loop equation
  analysis of the circular {$\beta$} ensembles},}\ }\href {\doibase
  10.1007/JHEP02(2015)173} {\bibfield  {journal} {\bibinfo  {journal} {J. High
  Energy Phys.}\ ,\ \bibinfo {pages} {173}} (\bibinfo {year}
  {2015})}\BibitemShut {NoStop}%
\bibitem [{\citenamefont {W{\"{u}}nsche}(2019)}]{Wu19}%
  \BibitemOpen
  \bibfield  {author} {\bibinfo {author} {\bibnamefont {W{\"{u}}nsche},
  \bibfnamefont {A.}},\ }\bibfield  {title} {\enquote {\bibinfo {title}
  {{Associated Hermite Polynomials Related to Parabolic Cylinder Functions}},}\
  }\href {\doibase 10.4236/apm.2019.91002} {\bibfield  {journal} {\bibinfo
  {journal} {Adv. Pure Math.}\ }\textbf {\bibinfo {volume} {09}},\ \bibinfo
  {pages} {15--42} (\bibinfo {year} {2019})}\BibitemShut {NoStop}%
\end{thebibliography}%

\end{document}